\documentclass[onefignum,onetabnum]{siamart190516}

\usepackage{lipsum}
\usepackage{amsfonts}
\usepackage{graphicx}
\usepackage{epstopdf}
\usepackage{algorithmic}
\usepackage{microtype}

\usepackage{amsmath,amsfonts}
\usepackage{url}
\usepackage{comment}

\usepackage{longtable}
\usepackage{caption}
\usepackage{multirow}
\usepackage{xcolor}
\usepackage{xspace}

\usepackage{graphicx}

\usepackage{caption}
\usepackage{subcaption}

\ifpdf
  \DeclareGraphicsExtensions{.eps,.pdf,.png,.jpg}
\else
  \DeclareGraphicsExtensions{.eps}
\fi

\newcommand{\bmat}[1]{\begin{bmatrix}#1\end{bmatrix}} %
\newcommand{\inv}{^{-1}}
\newcommand{\norm}[1]{\|#1\|}

\newcommand{\jbb}[1]{{\color{black} #1}} 
\newcommand{\jbbo}[1]{{\color{black} #1}} 

\definecolor{forestgreen}{rgb}{0.13, 0.55, 0.13}

\newcommand{\jbbR}[1]{{\color{black} #1}} 
\newcommand{\jbbRo}[1]{{\color{black} #1}} 
\newcommand{\jbbRt}[1]{{\color{black} #1}} 
\newcommand{\jbrv}[1]{{\color{black} #1}} 

\usepackage{amsmath,amsfonts}
\usepackage{multirow}
\usepackage{color}
\usepackage{algorithmic}
\usepackage{longtable,caption}
\usepackage{amsopn}
\usepackage{hyperref}
\usepackage{todonotes}

\usepackage{lscape}

\usepackage{longtable}

\newsiamremark{remark}{Remark}
\newsiamremark{hypothesis}{Hypothesis}
\crefname{hypothesis}{Hypothesis}{Hypotheses}
\newsiamthm{claim}{Claim}

\headers{PLSS: A Projected Linear Systems Solver}{J. J. Brust and M. A. Saunders} 

\title{PLSS: A Projected Linear Systems Solver%
\thanks{Version of \today. 
}}

\author{Johannes J. Brust%
   \thanks{Department of Mathematics,
   University of California San Diego, La Jolla, CA
   (\email{jjbrust@ucsd.edu}).}
\and Michael A. Saunders%
   \thanks{Department of Management Science and Engineering, Stanford University, Stanford, CA
   (\email{saunders@stanford.edu}).}
		}

\renewcommand{\b}[1]{\ensuremath{\mathbf{#1}}}

\newcommand{\bs}[1]{\ensuremath{{\boldsymbol{#1}}}}

\newcommand{\bko}[1]{\ensuremath{\mathbf{#1}_{k+1}}}
\newcommand{\bk}[1]{\ensuremath{\mathbf{#1}_k}}

\newcommand{\biidx}[2]{\ensuremath{{\mathbf{#1}}_{#2}}}

\newcommand{\btxiidx}[2]{\ensuremath{{\mathbf{#1}}_{\textnormal{#2}}}}
\newcommand{\bhiidx}[2]{\ensuremath{\widehat{{\mathbf{#1}}}_{#2}}} 
\newcommand{\biidxex}[3]{\ensuremath{{{\mathbf{#1}}^{\textnormal{#3}}_{#2}}}} 
\newcommand{\iidxex}[3]{\ensuremath{{{#1}^{\textnormal{#3}}_{#2}}}}

\newcommand{\To}{\!:\!}
\newcommand{\tp}{\ensuremath{^{\top}}}
\newcommand{\tpi}{\ensuremath{^{-\top}}}
\newcommand{\pseu}{\ensuremath{^{\dagger}}}

\newcommand{\A}{\mathbf{A}}
\newcommand{\B}{\mathbf{B}}

\newcommand{\I}{\mathbf{I}}
\newcommand{\Lbold}{\mathbf{L}}

\newcommand{\R}{\mathbf{R}}
\newcommand{\Sbold}{\mathbf{S}}
\newcommand{\Tbold}{\mathbf{T}}
\newcommand{\U}{\mathbf{U}}
\newcommand{\V}{\mathbf{V}}
\newcommand{\W}{\mathbf{W}}
\newcommand{\Y}{\mathbf{Y}}
\newcommand{\bbold}{\mathbf{b}}
\newcommand{\p}{\mathbf{p}}
\newcommand{\rb}{\mathbf{r}}
\newcommand{\s}{\mathbf{s}}
\newcommand{\x}{\mathbf{x}}
\newcommand{\y}{\mathbf{y}}
\newcommand{\z}{\mathbf{z}}
\newcommand{\e}{\mathbf{e}}
\newcommand{\Rk}{\R_k}
\newcommand{\Lk}{\Lbold_k}
\newcommand{\Sk}{\Sbold_k}
\newcommand{\Tk}{\Tbold_k}
\newcommand{\Yk}{\Y_k}
\newcommand{\Uk}{\U_k}
\newcommand{\Vk}{\V_k}
\newcommand{\pk}{\p_k}

\newcommand{\rk}{\rb_k}

\newcommand{\xk}{\x_k}

\newcommand{\yk}{\y_k}
\newcommand{\zk}{\z_k}
\newcommand{\ek}{\e_k}

\definecolor{darkgreen}{rgb}{.65, .5, 0}
\newcommand{\jjb}[1]{\textcolor{black}{#1}}

\begin{document}
	
	\maketitle
	
	\begin{abstract}
	We propose iterative \jbbR{projection} method\jbbRo{s} for solving square or rectangular consistent linear systems $\A\x = \bbold$.
	\jbbRo{Existing projection methods use sketching matrices (possibly randomized) to generate a sequence of small projected subproblems, but even the smaller systems can be costly. We develop a process that appends one column to the sketching matrix each iteration and converges in a finite number of iterations 
	whether the sketch is 
	random or deterministic. In general, our process generates orthogonal updates to the approximate solution $ \xk $. By choosing the sketch to be the set of all previous residuals, we obtain a simple recursive update and convergence in at most $\text{rank}(\A)$ iterations (in exact arithmetic). By choosing a sequence of identity columns for the sketch, we develop a generalization of the Kaczmarz method. In experiments on large sparse systems, our method (PLSS) \jbbRt{with residual sketches is competitive with
	LSQR and LSMR, and 
	with residual and identity sketches compares favorably with state-of-the-art randomized methods.}}
	\end{abstract}
	
	
	
	\section{Introduction}
	\label{sec:introduction}
	
	Consider a general linear system 
	\begin{equation} \label{eq:axb}
	   \A \x = \bbold,  \qquad \bbold \in \textrm{range}(\A),
	\end{equation}
	where $ \A \in \mathbb{R}^{m\times n} $, $ \x \in \mathbb{R}^n $ and $ \bbold \in \mathbb{R}^m $.
	For computations with large matrices, randomized methods  \cite{GowerRichtarik15,HMT11,RichtarikTakac20}
	aim to generate \jjb{inexpensive} (possibly low-accuracy) estimates of
	$\x$ by solving smaller projected systems. 
	When the elements of $\A$ are contaminated by errors or noise, \jbbRo{which} may be the
	case in data-driven problems, highly
	accurate solutions are not always requested. 
	For large systems, randomized solvers are growing in popularity, although it is not uncommon
	to observe slow convergence for naive implementations. \jbbRo{Reasons} for the widespread interest
	may be the arguably intuitive approach of solving a sequence of 
	small projected systems instead of \eqref{eq:axb}, \jbbRo{and the fact that randomization has emerged as an
	enabling technology in data science}.
	This article develops \jbbRo{a family of} projection methods 
	that solve a sequence of smaller systems
	\jbbRo{and can have} significant advantages in terms of computation, accuracy, and convergence.

    \subsection{Notation} \label{subsec:notation}
    \jbb{Integer $k \ge \jbbRt{1}$ represents the iteration index,
    and vector $\ek$ denotes the $k^{\text{th}}$ column of the identity matrix, with dimension
    depending on the context. \jbbR{The $ k^{\textnormal{th}} $ row of $ \b{A} $ is $ \bk{a}\tp = \ek\tp\b{A}  $}. For the $k^{\text{th}}$ solution estimate $\xk$, the residual vector is $\rk:=\bbold-\A\xk$, with associated vector $\yk:=\A\tp\rk$.
    Lower-case Greek letters represent scalars, and 
    the range of integers from $1$ to $k$ is written $1 \To k$.
    To prove convergence we make use of the economy SVD, $\A = \U\bs{\Sigma}\V\tp$.} 
	
	\subsection{Related work} \label{subsec:related_work}

	To handle large problems, there has been a growing interest in sketching techniques. 
	A \jjb{straightforward} approach is \emph{random sketching}, which consists of selecting a random matrix 
	$\Sbold \in \mathbb{R}^{m \times r}$ with $ r \ll m $ and solving the reduced linear system 
	\begin{equation} \label{eq:SaxSb}
	   \Sbold \tp \! \A \x = \Sbold \tp \bbold.
	\end{equation}
	By the Johnson-Lindenstrauss lemma \cite{JohnsonLindenstrauss84}, a solution to the reduced system \eqref{eq:SaxSb} is related to a solution
	of \eqref{eq:axb}. However, unless $\mathrm{range}(\Sbold)=\mathrm{range}(\A)$, the solutions will not be the \jjb{same}. 
	\jjb{Since \eqref{eq:SaxSb} involves the 
	product $\Sbold\tp \A$, previous work in \cite{ailon2009fast,Cartis2021,kane2014sparser} has focused on the choice of $\Sbold$ to reduce the computational cost of this product. 
	In particular, with certain  sketching matrices $\Sbold$, the randomized Kaczmarcz \cite{StrohmerVershynin09}, randomized coordinate descent \cite{leventhal2010randomized}, and stochastic
	Newton \cite{qu2016sdna} methods can be defined by \eqref{eq:SaxSb}.
	An overview of randomized iterative methods for solving
	linear systems is in \cite{GowerRichtarik15,RichtarikTakac20}, which we consider \jbbRo{to be} state-of-the-art for the purpose of this
	article.}
	
	
	
	
	\subsection{Motivation}
	\label{subsec:motivation}
	Given $\x_0 \in \mathbb{R}^n$,
	let estimates of the solution to \eqref{eq:axb} be defined by the iterative process
	\begin{equation}
	    \label{eq:xko}
	    \jbbR{\xk = \x_{k-1} + \jbbRt{\p_{k}}, \quad \quad k=1,2,\ldots,}
	\end{equation}
	where $ \jbbRt{\p_{k}} \in \mathbb{R}^n $ is called the \emph{update}. It is important to compute the update
	efficiently. Because sketching techniques aim to generate iterates with relatively low
	computational complexity, we choose to solve a sequence of sketched systems. In particular,
	we use a sequence of full-rank matrices 
	\begin{equation}
	\label{eq:seqS}
	\Sk \in \mathbb{R}^{m \times \jbbR{k}},\ k=\jbbR{1,2,}\dots, \qquad\textrm{rank}(\Sk) = \jbbR{k}.
	\end{equation}
	The matrices can be arbitrary as long as they have the specified dimensions and rank.
	We assume that $ \Sk $ is in the range of $\A$, though this is not strictly necessary.
	We show later that the
	iterates $\xk$ for solving \eqref{eq:axb} converge in a finite number of steps.
	We also demonstrate that a particular choice for $ \Sk $ results in very efficient
	updates.
	
	\jbbRo{Part of our scheme is an additional} \jbb{full-rank} parameter matrix $ \B\tp\B \in \mathbb{R}^{n \times n} $ that can
	possibly improve the numerical behavior of the methods. 
	(This matrix is not essential,
	and all results hold when $ \B = \b{I}_{n \times n} $.) Computations with $ \B\tp\B $
	are intended to be inexpensive, as they would be if it were a diagonal matrix. 
	The sequence of \jbbR{systems that define each update are of the form
	\begin{equation}
	    \label{eq:KKTSys}
	   \bmat{\B\tp \B & \A\tp \Sk
	      \\ \Sk\tp \A& \b{0}_{k \times k}}
          \bmat{\jbbRt{\p_{k}}\\ \jbbRt{\bs{\lambda}_{k}}}
    = \bmat{\A\tp \\ \Sk\tp}
      (\bbold-\A\x_{k-1})
	\end{equation}
	with $\jbbRt{\bs{\lambda}_k} \in \mathbb{R}^{k \times 1}$.
	Since only $\jbbRt{\p_{k}}$ is used to define the next iterates, 
	we do not compute $\bs{\lambda}_k$.
	When $\Sk$ is in the range of $\A$, solving \eqref{eq:KKTSys} is equivalent to the constrained least-squares problem
	\begin{align} 
	    \underset{ \p \in \mathbb{R}^n }{ \text{ arg min } } \quad & \frac{1}{2} \norm{\B\p}^2_2 
	   \label{eq:itProb}
	\\ 
	   \text{ subject to } \quad & \jbbR{\Sbold\tp_{k}}\A(\jbbR{\x_{k-1}} + \p) = \jbbR{\Sbold\tp_{k}} \bbold. 
	   \label{eq:cons}
	\end{align}
	Details of formulating problem \eqref{eq:itProb}--\eqref{eq:cons} from \eqref{eq:KKTSys} are in Appendix \ref{app:A}.}

	Here we summarize that the solution $ \jbbR{\jbbRt{\p_{k}}} $ to \eqref{eq:itProb}--\eqref{eq:cons} is defined by 
	$
	    \W := (\B\tp \B)\inv
	$ 
	and 
	\begin{equation}
	    \label{eq:pkLong}
	   \jbbR{\jbbRt{\p_{k}}} = \W\A\tp \jbbR{\Sbold_{k}} 
	    (\jbbR{\Sbold\tp_{k}} \A \W \A\tp \jbbR{\Sbold_{k}}  )\inv \jbbR{\Sbold\tp_{k}} (\bbold - \A\jbbR{\x_{k-1}}).
	\end{equation}
	If $ \Sk $ is not in the range of $ \A $, the inverse in \eqref{eq:pkLong}
	is replaced by the pseudo-inverse.
	For ease of notation we set $ \B\tp \B = \b{I} $ in the next sections (and hence $ \W = \b{I} $). Later we lift this assumption and describe updates with nontrivial $ \W $.
	
	Once $ \jbbR{\jbbRt{\p_{k}}} $ is obtained from \eqref{eq:pkLong}, it can define the next iterate via \eqref{eq:xko}.
	Note that popular methods in \cite{GowerRichtarik15,RichtarikTakac20} use an update like
	\eqref{eq:pkLong}; however, the matrix $ \Sbold \equiv \Sk $ is then typically randomly generated, and
	$ \A\tp \Sbold $, $ \Sbold \tp \A \A\tp \Sbold $ and $(\Sbold \tp \A \A\tp \Sbold)\inv$
	are typically recomputed each iteration. Thus, a small parameter $ r>0 $ must be selected
	such that a random $ \Sbold \in \mathbb{R}^{ m \times r } $ maintains low computational cost. Convergence is characterized by a rate $ \rho \in [0,1) $ that depends on the smallest singular value of a matrix defined by $ \Sbold $ and $\A$. Convergence implicitly depends on the choice of $r$ and may result in prohibitively many iterations when
	the rate $\rho$ is close to unity.
 	
 	\subsection{\jbbRo{Contributions}}
 	\label{subsec:contrib}

We prove that iteration \eqref{eq:xko}--\eqref{eq:pkLong} converges in $k_{\max}$
steps ($1 \le k_{\max} \le \min (m, n)$) when the sketch is in the range of $\A$ 
or the system is underdetermined. Each sketch $\Sk$ can be random
or deterministic, and can be augmented by one column each
iteration or recomputed from scratch. We show that the
finite-termination property holds for $m \ge n$ and also $m < n$. By
selecting previous residuals to form each $\Sk$, we
develop an iteration with orthogonal updates and residuals. This
process is simple to implement and only stores and updates five
vectors. By selecting columns of the identity matrix for each $\Sk$, we
develop an update that generalizes the Kaczmarz method \cite{Kaczmarz}. These
choices are two instances in our general class of methods
characterized by the choice of $\Sk$.

\section{Method}
\label{sec:method}

Our method solves a sequence of sketched systems
$\Sk\tp \A\x_k = \Sk\tp \bbold$ that exploit information generated at previous iterations.  Suppose that one sketching column $\s_k$ is generated each iteration and stored in the matrix
\begin{equation}
   \label{eq:genS}
   \Sk :=
   \bmat{\s_1 & \s_2 & \dots & \s_k}
   \in \mathbb{R}^{m \times \jbbR{k}}.
\end{equation}
Throughout, 
we assume that $ \Sk $ has full column rank.

	\subsection{Orthogonality}
		\label{subsec:orthog}
    \jbbRo{When constraints \eqref{eq:cons} are satisfied, each update
    has the property 
    \begin{equation}
    \label{eq:orthSr}
        \b{0} = \Sk\tp \A ( \jbbR{\x_{k-1} + \jbbRt{\p_{k}}} ) - \Sk\tp\bbold = -\Sk\tp \jbbR{\bk{r}}.
    \end{equation}
    Therefore from \eqref{eq:genS}, by construction, previous sketching columns are orthogonal to the next residual  (and hence linearly independent):
    $ \jbb{\b{s}_j} \perp \jbb{\b{r}_{i}} \textnormal{ for } \jbbRt{1} \le j \le \jbbRt{i}, 1 \le i \le \jbbRt{k} $. 
In section \ref{subsec:recpk} we show that the updates $\pk$ are also orthogonal.}

		
		
		

\subsection{Practical computations}
\label{subsec:practical}
\jbbRo{We develop general techniques to compute $ \jbbR{\p_{k}}$ \eqref{eq:pkLong}} efficiently. 
\jbbR{First, we describe a method based on updating a QR factorization.}
\jbbR{Second}, we deduce an \jbbR{alternative} method 
(based on updating a triangular factorization) to avoid recomputing $\jbbR{\Sbold\tp_{k}} \A$ and $(\jbbR{\Sbold\tp_{k}} \A \A\tp \jbbR{\Sbold_{k}} )\inv$. \jbbR{However, since both of these methods, \jbbRo{for general sketches}, have memory requirements
that grow with $k$}, we \jbbR{additionally} show \jbbRo{in Section \ref{subsec:recpk}} that $\p_{\jbbRt{k}}$ satisfies
a short recursion defined by $\p_{k-1}$ and another vector \jbbRo{when the sketch is chosen judiciously}. 

In order to avoid recomputing the potentially expensive product $\Sk\tp \A$, 
let $ \jbbR{\jbbRt{\y_{k}}} := \A\tp\jbbRo{\jbbRt{\s_{k}}} $ and define
$ \biidx{Y}{k-1} := [\jbbRt{\: \y_1 \: \y_2 \: \dots \: \y_{\jbbRt{k-1}}}] $ to collect the previous $\y$'s. Then
\begin{equation}
\label{eq:yk}
   (\Sk\tp\A)\tp = \A\tp \Sk =
   \bmat{\A\tp \biidx{S}{k-1} & \A\tp \jbbRo{\jbbRt{\s_k}}} =
   \bmat{\biidx{Y}{k-1} & \jbbR{\jbbRt{\y_{k}}}} = \bk{Y} \in \mathbb{R}^{n \times \jbbR{k}}.
\end{equation}

\subsection{\jbb{QR factorization}}
\label{sec:QR}
In general it is numerically safer to update factors of a matrix rather than its inverse (because factors exist even when the matrix is singular). QR
factors of $\Yk$ can be used to this effect. Specifically, let the QR factorization be 
\begin{equation}
    \label{eq:ykeqQR}
    \Yk=\bk{Q}\Tk, 
\end{equation}
where $\bk{Q}\in\mathbb{R}^{n\times \jbbR{k}}$ is orthonormal and $ \bk{T}\in \mathbb{R}^{\jbbR{k}\times \jbbR{k}} $ is upper triangular.
Note that the update $\p_{\jbbRt{k}}$ from \eqref{eq:pkLong} with $\rho_{\jbbR{k-1}}\equiv \jbbRo{\jbbRt{\s_{k}\tp} \rb_{\jbbR{k-1}}}$ 
simplifies to 
\begin{equation}
    \label{eq:pkQR}
    \p_{\jbbRt{k}}=\Yk(\Yk\tp\Yk)\inv\Sk\tp\rb_{\jbbR{k-1}}=\frac{\rho_{\jbbR{k-1}}}{\jbbRt{r_{kk}}}\bk{Q}\bk{e}=\frac{\rho_\jbbR{k-1}}{\jbbRt{\jbbRt{r_{kk}}}}\bk{q},
\end{equation}
where $\jbbRt{\jbbRt{r_{kk}}}$ is the final diagonal element in $\bk{T}$. Householder reflectors $\b{H}_{\jbbRt{k}}$ can be used to represent
$\bk{Q}=\biidx{H}{1}~\biidx{H}{2}~\cdots~\biidx{H}{\jbbR{k}}$ in factored form with essentially the same storage as $\Yk$ \cite{GVL96}.
Computing 
$\p_{\jbbRt{k}}$ in \eqref{eq:pkQR} requires one product with the factored form of $\bk{Q}$. In particular, $\b{q}_{\jbbR{k}}$ and $r_{\jbbR{k}\jbbR{k}}$ 
are obtained by
\begin{equation*}
    \bk{Q}\bk{e}=\b{q}_{\jbbR{k}} \quad \text{ and } \quad \bk{e}\tp(\bk{Q}\tp\b{y}_{\jbbRt{k}})= \b{q}_{\jbbR{k}}\tp\b{y}_{\jbbRt{k}}=r_{\jbbR{k}\jbbR{k}}.
\end{equation*}
We note that another option to develop the QR factorization \eqref{eq:ykeqQR} is to append $ \biidx{y}{k} $ to
$\biidx{T}{k-1}$ and apply a sequence of plane rotations to eliminate the nonzeros in $ (\biidx{y}{k})_{k+1:m} $
to define $\biidx{Q}{k}$ and $\biidx{T}{k}$.
Both QR strategies use storage that grow with $k$.


\subsection{Triangular factorization}
\label{subsec:triang}
\jbbR{In another approach, we update the inverse of $ \bk{Y}\tp \bk{Y} $ in \cref{eq:pkLong}}.
This is based on storing and updating the previous inverse in order to obtain the next.
Concretely, suppose we store 
$$ 
\biidx{N}{k-1} = (\biidx{Y}{k-1}\tp\biidx{Y}{k-1})\inv = (\Sbold\tp_{k-1} \A \A\tp \Sbold_{k-1})\inv
$$ 
and wish to compute $ \bk{N} $.

\begin{theorem}
	\label{thrm:simplified-inverse-full-history}
	Assume $\bk{Y}$ has full rank. Then
	\begin{equation}
	\label{eq:rdrt-decomposition-of-the-inverse}
	\bk{N} 
	=
	\bmat{\biidx{N}{k-1} + \frac{1}{ \delta_k} \bhiidx{\jbbRt{t}}{k}\bhiidx{\jbbRt{t}}{k}\tp& \frac{-1}{ \delta_{\jbbRt{k}}}\bhiidx{\jbbRt{t}}{k}
	\\ \frac{-1}{ \delta_{\jbbRt{k}}}\bhiidx{\jbbRt{t}}{k}\tp &
	   \frac{1}{\delta_{\jbbRt{k}}}}
	= \Rk \bk{D} \Rk\tp,
\end{equation}
where $ \Rk := [\: \biidx{\jbb{t}}{1}^{(k)} \: \biidx{\jbb{t}}{2}^{(k)} \: \cdots \: \biidx{\jbb{t}}{\jbbR{k}}^{(k)}]  $ is upper triangular, 
	$ \bk{D} := \textrm{diag}( 1/\delta_j )_{j=1,\dots,\jbbRt{k}}$, and 
\begin{equation*}
	   \delta_{j} := \y_j\tp \y_j - \y_j\tp \jbrv{\biidx{Y}{j-1}} \bhiidx{\jbb{t}}{\jbbR{j}}, \quad \biidx{\jbb{t}}{j}^{(k)} := 
	   \bmat{\bhiidx{\jbb{t}}{\jbbR{j}}
	      \\ -1
	      \\ \b{0}_{1:k-j}},
	\quad
	\text{ with }
	\quad
	   \bhiidx{\jbb{t}}{\jbbR{j}} := \jbrv{\biidx{N}{\jbbR{j-1}} \biidx{Y}{\jbbR{j-1}}\tp} \y_j.
\end{equation*}
\end{theorem}

\begin{proof}
Observe that
\begin{equation} \label{eq:Nk}
  \bk{N} = (\biidx{Y}{k}\tp\biidx{Y}{k})\inv =
  \bmat{\biidx{Y}{k-1}\tp\biidx{Y}{k-1} &  \biidx{Y}{k-1}\tp \jbbR{\jbbRt{\y_{k}}}
     \\ \jbbR{\y\tp_{k}} \biidx{Y}{k-1} & (\jbbR{\y\tp_{k}} \jbbR{\jbbRt{\y_{k}}})\inv}\inv
     =
  \bmat{\biidx{N}{k-1}\inv &  \biidx{Y}{k-1}\tp \jbbR{\jbbRt{\y_{k}}}
     \\ \jbbR{\y\tp_{k}} \biidx{Y}{k-1} & (\jbbR{\y\tp_{k}} \jbbR{\jbbRt{\y_{k}}})\inv}\inv.
\end{equation}
Setting
	$  \bhiidx{\jbbRt{t}}{k} := \biidx{N}{k-1} \biidx{Y}{k-1}\tp \jbbR{\jbbRt{\y_{k}}} $ 
	and 
	$ \jbbR{\delta_k} := \jbbR{\y\tp_{k}} \jbbR{\jbbRt{\y_{k}}} - \jbbR{\y\tp_{k}} \biidx{Y}{k-1} \bhiidx{\jbbRt{t}}{k} $ 
and then explicitly forming the block inverse in \cref{eq:Nk}, we obtain
\begin{equation} \label{eq:Nkrec}
	\bk{N} =
	\bmat{
	\biidx{N}{k-1} + \frac{1}{ \jbbR{\delta_k}} \bhiidx{\jbbRt{t}}{k}\bhiidx{\jbbRt{t}}{k}\tp& \frac{-1}{ \jbbR{\delta_k}}\bhiidx{\jbbRt{t}}{k} \\ \frac{-1}{\jbbR{\delta_k}}
	   \bhiidx{\jbbRt{t}}{k}\tp & \frac{1}{\jbbR{\delta_k}}},
\end{equation}
which proves the first equality.
	
	Since all terms on the right side of \cref{eq:Nkrec} depend on $ \biidx{N}{k-1} $ only,
	this recursion can be used to form $ \bk{N} $ once $ \biidx{N}{k-1} $ (and $\biidx{Y}{k-1}$ and $\jbbR{\jbbRt{\y_{k}}}$) have been stored.
	The recursive process can be initialized with $ \biidx{N}{0} := (\biidx{y}{1}\tp \biidx{y}{1})\inv $.
	Moreover, with the vectors
	\begin{equation*}
	   \biidx{\jbb{t}}{j}^{(k)} := 
	   \bmat{\bhiidx{\jbb{t}}{\jbbR{j}}
	      \\ -1
	      \\ \b{0}_{1:\jbbR{k}-j}},
	   \quad j = 1 \To \jbbR{k},
	\end{equation*}
	recursion
	\cref{eq:Nkrec} can be expressed as a sum of rank-one updates.  In particular,
	\begin{align*}
	\bk{N} &= \bmat{\biidx{N}{k-1} &
	             \\                & 0} 
 + \frac{\biidx{\jbb{t}}{\jbbR{k}}^{(k)}(\biidx{\jbb{t}}{\jbbR{k}}^{(k)})\tp}{\delta_{\jbbRt{k}}}
 \\ &= \bmat{\bmat{\biidx{N}{k-2} \\ & 0}
          \\ & 0}
   + \frac{\biidx{\jbb{t}}{k-1}^{(k)}(\biidx{\jbb{t}}{k-1}^{(k)})\tp}
          {\delta_{k-1}}
   + \frac{\biidx{\jbb{t}}{\jbbR{k}}^{(k)}(\biidx{\jbb{t}}{\jbbR{k}}^{(k)})\tp}
          {\delta_{\jbbRt{k}}}
 \\	&=\bmat{\bmat{\bmat{\biidx{N}{0} \\ & \ddots}
	                 \\	                & 0}
         \\ 	& 0}
     + \frac{\biidx{\jbb{t}}{1}^{(k)}(\biidx{\jbb{t}}{1}^{(k)})\tp}
            {\delta_1}
     + \cdots + \frac{\biidx{\jbb{t}}{\jbbR{k}}^{(k)}(\biidx{\jbb{t}}{\jbbR{k}}^{(k)})\tp}
          {\delta_{\jbbRt{k}}}. 
	\end{align*}
	Defining the upper triangular matrix 
	$ \Rk = [\: \biidx{\jbb{t}}{1}^{(k)} \: \biidx{\jbb{t}}{2}^{(k)} \: \cdots \: \biidx{\jbb{t}}{\jbbR{k}}^{(k)}]  $
	and the diagonal matrix 
	$ \bk{D} = \text{diag}( 1/\delta_j )_{j=1:\jbbR{k}} $ then gives the factorized representation of $ \bk{N} $ in \eqref{eq:rdrt-decomposition-of-the-inverse}.
\end{proof}

This factorization of $ \bk{N} $ already improves computing $ \p_{\jbbRt{k}} $ from \eqref{eq:pkLong}.
Specifically, a formula that does not require any solves is described in the following corollary.


\begin{corollary}\label{co:update-alpha1}
	If $ \bk{N} $ is generated by the process in Theorem \ref{thrm:simplified-inverse-full-history}, then 
	\begin{equation}\label{eq:compute-pk-alpha1}
	\begin{aligned}
	\p_{\jbbRt{k}} &= \frac{\jbbRo{\jbbRt{\s_{k}}\tp \rb_{k-1}}}{\delta_{{\jbbRt{k}}}}\left( \y_{{\jbbRt{k}}} - \Y_{k-1}\bhiidx{\jbbRt{t}}{k}   \right) \\
	&= \frac{\jbbRo{\jbbRt{\s_{k}}\tp \rb_{k-1}}}{\delta_{{\jbbRt{k}}}}\left( \y_{{\jbbRt{k}}} - \Y_{k-1}\biidx{R}{k-1}\biidx{D}{k-1}\biidx{R}{k-1}\tp \biidx{Y}{k-1}\tp\y_{{\jbbRt{k}}}    \right).
	\end{aligned}
	\end{equation}
\end{corollary}

\begin{proof}
	Substituting \eqref{eq:Nkrec} into \eqref{eq:pkLong} with the definition
	from \eqref{eq:yk} yields
	\begin{equation}\label{eq:compute-pk-alphaneq1}
	\begin{aligned}
	\b{p}_{\jbbRt{k}} =& \left(\Y_{k-1}\b{N}_{k-1} + \frac{1}{\delta_{{\jbbRt{k}}}}\Y_{k-1}\bhiidx{\jbbRt{t}}{k}\bhiidx{\jbbRt{t}}{k}\tp - \frac{1}{\delta_{{\jbbRt{k}}}}\y_{{\jbbRt{k}}}\bhiidx{\jbbRt{t}}{k}\tp  \right)\Sbold_{k-1}\tp\rb_{{\jbbR{k-1}}} \\
	&+ \frac{\jbbRo{\jbbRt{\s_{k}}\tp \rb_{k-1}}}{\delta_{{\jbbRt{k}}}}\left( \y_{{\jbbRt{k}}} - \Y_{k-1}\bhiidx{\jbbRt{t}}{k}   \right),
	\end{aligned}
	\end{equation}
	where $\Sbold_{k-1}\tp\rb_{{\jbbR{k-1}}}=\b{0}$ by the orthogonality \jbbRo{property.} 
\end{proof}

Formula \eqref{eq:compute-pk-alpha1} implies that we do not have to compute $ \bk{Y}\tp \bk{Y} $
 explicitly, nor do solves with it. Instead, the factors
$ \R_{{\jbbR{k-1}}} $ and $ \b{D}_{{\jbbR{k-1}}} $ can be updated \jbb{one column per iteration} by products with triangular
matrices only. 
\jbb{For instance, $ \R_{{\jbbR{k-1}}} $ is obtained by computing $ \biidx{t}{\jbbR{k-1}} $ (\jbbR{suppressing superscripts}) and appending: $ \R_{{\jbbR{k-1}}} = [ \: \biidx{R}{k-2} \quad \biidx{t}{\jbbR{k-1}} \: ] $.
(Since $\biidx{t}{\jbbR{k-1}} = [ \bhiidx{t}{k-1}\tp \: {-1} \: \b{0}\tp ]\tp $, it is obtained from 
2 multiplications with triangular matrices only: $ \bhiidx{t}{k-1} = \biidx{R}{k-2}(\biidx{D}{k-2} (\biidx{R}{k-2}\tp ( \biidx{Y}{k-2}\tp \y_{\jbbR{k}} )))  $.)}
However, \eqref{eq:compute-pk-alpha1} still needs 
$ \biidx{Y}{k-1} $, $ \biidx{R}{k-1} $ and $ \biidx{D}{k-1} $, which all grow with $k$.


\subsection{ \jbbRo{Orthogonality of $ \p_{\jbbRt{k}} $}}
\label{subsec:orthP}

It is valuable to note from \eqref{eq:compute-pk-alpha1} that
\begin{equation*}
    \p_{\jbbRt{k}} 
    = \frac{\jbbRo{\jbbRt{\s_{k}}\tp \rb_{k-1}}}{\delta_{\jbbRt{k}}}\left(\b{I} - \biidx{Y}{\jbbR{k-1}}(\biidx{Y}{\jbbR{k-1}}\tp \biidx{Y}{\jbbR{k-1}})\inv \Y_{\jbbR{k-1}}\tp  \right)\b{y}_{\jbbRt{k}},
\end{equation*}
so that $ \biidx{Y}{k-1}\tp \p_{\jbbRt{k}} = \b{0} $.
Defining 
$\biidx{P}{\jbbR{k-1}} := \bmat{~\biidx{p}{1} & \dots & \biidx{p}{k-1}~}$
and noting that
$ \biidx{p}{i} \in \textnormal{span}(\biidx{Y}{i}) $ for
$ i = 1 \To k-1$, we see that 
\begin{equation}
    \label{eq:orthogonal-updates}
    \biidx{P}{k-1}\tp \p_{\jbbRt{k}} = \b{0}.
\end{equation}
That is, the updates generated by our \jbbRo{class of methods} are orthogonal. We also define the 
squared lengths $ \theta_i = \norm{\biidx{p}{i}}_2^2 $  
($ i = 1 \To \jbbR{k} $) and the diagonal matrix 
\begin{align*}
\bs{\Theta}_{k-1} :=
\biidx{P}{k-1}\tp \biidx{P}{k-1} &= \text{diag}(\theta_1,\ldots,\theta_{k-1}).
\end{align*}

\subsection{\jbbRo{  Linear combination of $ \jbbRt{\p_{k}} $}}
\label{subsec:lincomb}
\jbbR{The methods in Sections \ref{subsec:triang} and \ref{sec:QR} use memory that grows with $k$.} By further unwinding recursive relations in \eqref{eq:compute-pk-alpha1}, 
described in Appendix \ref{app:B}, we can represent the update as a linear combination
of previous updates. \jbbRo{This enables us to derive in Section \ref{subsec:recpk}} \jbbR{a short recursion defined in terms of $\biidx{p}{k-1}$ and $\biidx{y}{k}$ only}. For some scalars
$ \alpha_{j} $, the dependencies become 
\begin{equation}
    \label{eq:pkLincomb}
    \p_{\jbbRt{k}} = \frac{\jbbRo{\jbbRt{\s_{k}}\tp \rb_{k-1}}}{\delta_{\jbbRt{k}}} \big( \sum_{j=1}^{\jbbR{k-1}} \alpha_{j} \biidx{p}{j} 
        \jbb{+}\b{y}_{\jbbR{k}} \big).
\end{equation}

Representation \eqref{eq:pkLincomb} implies that $ \jbbRt{\p_{k}} = \b{P}_{k-1} \b{g}_{k-1} + \gamma_{k-1} \jbbRt{\y_{k}} $ for
some vector $ \b{g}_{k-1} \in \mathbb{R}^{k-1} $ and scalar $ \gamma_{k-1} $. This leads to a computationally
efficient formula for $ \p_{\jbbRt{k}} $.


\section{Convergence}
 	\label{subsec:conv}
 	We first show finite termination for iterates generated by 
 	\eqref{eq:xko} and \eqref{eq:pkLong} using a sequence $\{\Sbold_i \in \mathbb{R}^{m \times \jbbR{i}}\}_{\jbbR{i=1}}^\jbbR{k}$ in the range of 
 	$\A$. \jbrv{This ensures that update \eqref{eq:pkLong} is well defined in terms of the inverse.} Denote by $(\cdot)\inv$ the inverse for square matrices
 	and by $(\cdot)^{\dagger}$ the pseudo-inverse for rectangular matrices.

\begin{theorem}
    \label{thrm:1}
	Assume that $m \geq n$ and $\A$ has full column rank. 
	Given $\x_0 \in \mathbb{R}^n$, consider the sequence $\{\xk \}$ computed by \eqref{eq:xko} and
	\eqref{eq:pkLong}, \jbrv{where the inverse in \eqref{eq:pkLong} is well defined}. 
	Also assume that after $\jbbR{n}$ iterations, $\jbbR{\Sbold_{n}}$ has full rank and is in the range of $ \A $. 
	Then, $\x_{n}$ solves \eqref{eq:axb}.
\end{theorem}

\begin{proof}
	After $k = \jbbR{n}$ iterations, 
	$\Sbold_k$ is by assumption a rank-$n$ matrix in the range of $\A$, so that $\Sbold_k \tp\A$ is a nonsingular square matrix. 
		Hence, $\left( \Sbold_k \tp\A \W \A\tp\Sbold_k \right)\inv=(\A \tp\Sbold_k)\inv \W\inv (\Sbold_k \tp \A)\inv$. Let $\jbbR{\rb_{k-1}} := \bbold - \A \jbbR{\x_{k-1}}$, so that
	\begin{equation*}
		\jbbR{\jbbRt{\p_{k}}}= \W \A\tp\Sbold_k \left( \Sbold_k\tp\A \W \A\tp\Sbold_k \right)\inv \Sbold_k\tp \jbbR{\rb_{k-1}} = 
		(\Sbold_k\tp\A)\inv \Sbold_k\tp\jbbR{\rb_{k-1}}. 
	\end{equation*}
	Let $\A = \U \bs{\Sigma} \V\tp$ be the economy SVD of $\A$.
	As $ \Sk $ is in the range of $ \A $, it can be represented by $ \Sk = \U\Tk $ for some nonsingular $\Tk$. Therefore,
	\begin{align*}
	    (\Sbold_k\tp\A)\inv \Sbold_k\tp\jbbR{\rb_{k-1}} &= (\Tk\tp\bs{\Sigma}\V\tp)\inv \Tk\tp\U\tp \jbbR{\rb_{k-1}}
	    = \A\pseu \jbbR{\rb_{k-1}},
	\\ \x_{n} = \x_{n-1} + \p_{n} &=  \x_{n-1} + \A\pseu(\jbb{\bbold - \A \x_{n-1}}) = \A\pseu\bbold.
	\end{align*}
	We conclude that $ \x_n $ is the least squares solution of \eqref{eq:axb}
	\jbb{when $m > n$, and the unique solution when $m =n$.}
\end{proof}

When not every member of $ \{ \Sbold_k \} $ is in the range 
of $ \A $, Corollary \ref{co:1} shows convergence in at most $m$ iterations.

\begin{corollary}
    \label{co:1}
    Suppose we use the same iterative process \jbbRo{from} Theorem \ref{thrm:1} except that
    each matrix in the sequence $ \{ \Sk \}_{\jbbR{k=1}}^{\jbbR{m-1}} $ has full rank only
    (and is not necessarily in the range of $ \A $), $ \jbbR{\Sbold_{m}} $
    is a square nonsingular matrix, \jbrv{and \eqref{eq:pkLong} is defined by the pseudo-inverse}. Then \jbbRo{either $ \x_m $ or $\x_l$, \jbrv{$n \le l < m$}, is a} 
    solution of \eqref{eq:axb}.
\end{corollary}

\begin{proof}
    At iteration $ k = \jbbR{m} $, the matrix $ \Sk\tp \A \W \A\tp \Sk \in \mathbb{R}^{m \times m} $ does not have full rank. Thus the update $ \pk $ in \eqref{eq:pkLong} is defined by the pseudo-inverse
    $ ( \Sbold_k\tp\A \W \A\tp\Sbold_k)\pseu $. \jbbRo{Since} $ \Sk $ is square and
    nonsingular, 
    \begin{equation*}
        \Sk( \Sbold_k\tp\A \W \A\tp\Sbold_k)\pseu\Sk\tp =
        (\A \W \A\tp)\pseu.
    \end{equation*}
    Let $\jbbR{\rb_{k-1} := \bbold - \A\jbbR{\x_{k-1}}}$ so that 
    $ \pk = \W \A\tp (\A \W \A\tp)\pseu \jbbR{\rb_{k-1}} =  \A\pseu \jbbR{\rb_{k-1}} $.
    Then 
    $$
      \x_{m} = \x_{m-1} + \p_{m} =  \x_{m-1} + \A\pseu(\jbb{\bbold - \A \x_{m-1}}) = \A\pseu\bbold.
    $$
    At an earlier iteration \jbrv{$n \le k = l < \jbbR{m}$}, if $\biidx{S}{k}$ can be
    partitioned by a square nonsingular matrix $ \biidx{T}{k} \in \mathbb{R}^{n \times n} $
    and a matrix $ \U_{\perp} \in \mathbb{R}^{m \times (\jbbR{k}-n)} $ in the nullspace of $ \jbbR{\A\tp} $ (i.e., $ \U\tp_{\perp} \U = \b{0} $ where
    $ \A = \U \bs{\Sigma} \V\tp $) so that
    \begin{equation*}
        \biidx{S}{k} = 
        \bmat{\U \biidx{T}{k} & \U_\perp},
    \end{equation*}
    then $ \biidx{x}{\jbbR{l}} $ is a solution to \eqref{eq:axb}. In particular,
    \begin{equation*}
      (\biidx{S}{k}\tp \A \W \A\tp \Sk)\pseu =
       \bmat{\Tk \tp\bs{\Sigma} \V\tp \W \V \bs{\Sigma} \Tk  & \b{0}
       \\ \b{0} & \b{0}}^{\dagger}.
    \end{equation*}
    The update is then
    \begin{equation*}
      \jbbR{\jbbRt{\p_{k}}} = \bmat{\A\pseu & \b{0}}
            \bmat{\b{I}_m \\ \b{0}} \jbbR{\rb_{k-1}},
    \end{equation*}
    and hence $\jbbR{\xk} = \jbbR{\x_{k-1}} + \jbbR{\jbbRt{\p_{k}}}$ will be a least-squares solution of \eqref{eq:axb}.
\end{proof}

\jbb{Corollary \ref{co:1} further implies that the process in \eqref{eq:xko} and
\eqref{eq:pkLong}, with appropriate sketching matrices,
also finds a solution when $ \A $ is underdetermined.}

\begin{corollary}
    \label{co:2}
    \jbb{If $ m <n $ so that $ \A \in \mathbb{R}^{m \times n} $ is underdetermined,
    and $\{ \Sk \}_{\jbbR{k=1}}^{\jbbR{m-1}} $ is a sequence of full-rank matrices with 
    $ \jbbR{\Sbold_{m}} $ nonsingular, then $ \x_m $ solves  \eqref{eq:axb}.}
\end{corollary}

\begin{proof}
    \jbb{At iteration $ k = \jbbR{m} $, matrix $ \jbbR{\Sbold_{m}} $ is square and nonsingular so that
    the update is given by (\jbbRo{as} in Corollary \ref{co:1})
    \begin{equation*}
        \jbbR{\jbbRt{\p_{k}}} = \W \A\tp (\A \W \A\tp)\pseu \jbbR{\rb_{k-1}},
        \qquad \jbbR{\rb_{k-1}} = \bbold - \A \jbbR{\x_{k-1}}.
    \end{equation*}
    Therefore
    \begin{align*}
        \A \x_m   &= \A(\x_{m-1} + \p_{m} ) \\
                        &= \A\x_{m-1} + \A \W \A\tp (\A \W \A\tp)\pseu \b{r}_{m-1} \\
                        &= \A\x_{m-1} + (\bbold - \A\x_{m-1}),
    \end{align*}
    and we conclude that $ \A \x_m = \bbold $ and \jbbRo{that} $ \x_m $ solves \eqref{eq:axb}.}
\end{proof}

\section{PLSS residuals} 
\label{subsec:recpk}
	\jbbRo{By storing a history of residuals in the sketching matrix, we construct an iteration
	with orthogonal residuals and updates}. Recall that
	$\jbb{\rk := \bbold - \A \xk}$ 
	is the residual at iteration $k$ (where since $\bbold$ is in the range of $\A$,
	all residuals are, too). Suppose the previous residuals have been stored, i.e., $ \s_i = \rb_{i-1} $ ($ 1 \le i \le k $) so that the history of all residuals is in the matrix
	\begin{equation}
		\label{eq:S}
	\Sk := \left[\ \jbb{\b{r}_0 \quad \b{r}_1 \quad \cdots \quad \jbbR{\rb_{k-1}}} \ \right] \in \mathbb{R}^{m \times \jbbR{k}}.
	\end{equation}
    \jbbRo{Since $ \Sbold_{k} \tp \rb_k = \b{0} $, all residuals are orthogonal with this choice of sketch. Moreover, since all residuals are in the
    range of $\A$, by Theorem \ref{thrm:1}, iteration \eqref{eq:xko} converges in at most $\text{min}(m,n)$ iterations
    in exact arithmetic. Defining the scalar $ \rho_i = \norm{\rb_i}^2_2 $ ($ 0 \le i \le k-1 $) we develop a 1-step
    recursive update.}

\begin{theorem}
    \label{thrm:2}
	The update $\b{p}_{\jbbRt{k}}$ from \eqref{eq:pkLong} with $ \Sk $ from \eqref{eq:S} 
	can be computed by the 1-step recursive formula
	\begin{align}
	    \label{eq:pkShort}
	    \b{p}_{\jbbRt{k}} &= \beta_{{\jbbR{k-1}}}\biidx{p}{{\jbbR{k-1}}}+\gamma_{{\jbbR{k-1}}}\b{y}_{{\jbbRt{k}}},
	\\ \beta_{{\jbbR{k-1}}} &:= 
\frac{1}
       {\big(\frac{\norm{\biidx{p}{k-1}} \norm{\b{y}_{{\jbbRt{k}}}}}{\norm{\biidx{r}{k-1}}\norm{\biidx{r}{k-1}}} - 1\big)
        \big(\frac{\norm{\biidx{p}{k-1}} \norm{\b{y}_{{\jbbRt{k}}}}}{\norm{\biidx{r}{k-1}}\norm{\biidx{r}{k-1}}} + 1\big)},
  \\ \gamma_{\jbbR{k-1}} &:=
           \frac{1}{\norm{\biidx{y}{k}}^2 \big(1-\frac{\norm{\rb_{\jbbR{k-1}}}\norm{\rb_{\jbbR{k-1}}}}{\norm{\biidx{p}{k-1}}\norm{\biidx{y}{k}}}\big)
           \big(1+\frac{\norm{\rb_{\jbbR{k-1}}}\norm{\rb_{\jbbR{k-1}}}}{\norm{\biidx{p}{k-1}}\norm{\biidx{y}{k}}}\big)}.
	\end{align}
\end{theorem}

\begin{proof}
	From \eqref{eq:pkLincomb}, $ \p_{\jbbRt{k}} $ can be represented as a linear combination
	of the columns in $ \biidx{P}{k-1} $ and $ \b{y}_{\jbbR{k}} $. Thus for a vector
	$ \biidx{g}{k-1} \in \mathbb{R}^{k-1} $ \jbb{and scalar $ \gamma_{k-1} $} we have
	\begin{equation}
	    \label{eq:pkPkm1}
	    \p_{\jbbRt{k}} = \biidx{P}{k-1} \biidx{g}{k-1} + \gamma_{k-1} \b{y}_{\jbbR{k}}. 
	\end{equation}
	Since $ \biidx{P}{k-1}\tp \p_{\jbbRt{k}} = \b{0} $, by multiplying
	\eqref{eq:pkPkm1} left and right by $ \biidx{P}{k-1}\tp $
	and solving with the diagonal $ \bs{\Theta}_{\jbbR{k-1}} $ we obtain
	 $\b{g}_{\jbbR{k-1}} = - \gamma_{\jbbR{k-1}} \bs{\Theta}_{\jbbR{k-1}}\inv \biidx{P}{k-1}\tp \b{y}_{\jbbR{k}}$.
	To simplify $ \biidx{P}{k-1}\tp \b{y}_{\jbbR{k}} $, recall that $ \b{\jbb{r}}_{\jbbR{k-1}} = \bbold - \A\x_{\jbbR{k-1}} $ and $ \x_{\jbbR{k-1}} = \biidx{x}{k-2} + \biidx{p}{k-1} $, so	that 
	\begin{equation}
	    \label{eq:recRes}
	    \rb_{\jbbR{k-1}} = \biidx{r}{\jbbR{k-2}} - \A \biidx{p}{\jbbRt{k-1}}.
	\end{equation}
	Because residuals are orthogonal, multiplying \eqref{eq:recRes} on the left by
	$ \Sk\tp $ gives
	\begin{equation*}
	    \rho_{\jbbR{k-1}} \biidx{e}{k} = \rho_{k-2} \biidx{e}{k-1} - \bk{Y}\tp \biidx{p}{k-1}
	    \quad \text{ and } \quad \rho_{k-1} = -\b{y}_{\jbbR{k}}\tp \biidx{p}{k-1}.
	\end{equation*}
	From \eqref{eq:recRes} we also have
	\begin{equation*}
	    \rb_{\jbbR{k-1}} = \biidx{r}{k-3} - \A \biidx{p}{k-2} - \A \biidx{p}{k-1}, 
	    \quad 
	    \rho_{k-1} \biidx{e}{k} = \rho_{k-3} \biidx{e}{k-2} - \bk{Y}\tp \biidx{p}{k-2} - \bk{Y}\tp \biidx{p}{k-1},
	\end{equation*}
	and $ \b{y}_{\jbbRt{k}}\tp \biidx{p}{k-1} = -\rho_{\jbbR{k-1}} $ implies that 
	$ \b{y}_{\jbbRt{k}}\tp \biidx{p}{k-2} = 0 $. Similarly, $ \b{y}_{\jbbRt{k}}\tp \biidx{p}{k-i} = 0  $ for
	$ i =3 \To k-1 $. Therefore, $ \biidx{P}{k-1}\tp \b{y}_{\jbbRt{k}} = -\rho_{\jbbR{k-1}} \b{e}_{\jbbR{k-1}} $ and
	    $\biidx{g}{k-1} = (\gamma_{\jbbR{k-1}}\rho_{\jbbR{k-1}}/\theta_{k-1}) \b{e}_{\jbbR{k-1}}$.
	From \eqref{eq:pkLincomb},
	\begin{equation}
	    \label{eq:pkIntermed}
	    \p_{\jbbRt{k}} =  \frac{\gamma_{\jbbR{k-1}}\rho_{\jbbR{k-1}}}{\theta_{k-1}} \biidx{p}{k-1} + \gamma_{\jbbR{k-1}} \y_{\jbbR{k}},
	\end{equation}
	and from \eqref{eq:pkLong} we have $ \b{y}_{\jbbRt{k}}\tp \p_{\jbbRt{k}} = \Sk\tp \rb_{\jbbR{k-1}} =  \rho_{\jbbR{k-1}} $. Thus, multiplying
	\eqref{eq:pkIntermed} by $ \b{y}_{\jbbRt{k}}\tp $ gives 
	\begin{equation*}
	    \rho_{\jbbR{k-1}} = - \frac{\gamma_{\jbbR{k-1}} \rho_{\jbbR{k-1}}^2}{\theta_{k-1}} + \gamma_{\jbbR{k-1}} \norm{\b{y}_{\jbbRt{k}}}^2.
	\end{equation*}
	Solving the previous expression for $ \gamma_{\jbbR{k-1}} $ we find
\begin{align*}
   \gamma_{\jbbR{k-1}} &= \frac{\theta_{k-1}\rho_{\jbbR{k-1}}}{\theta_{k-1}\norm{\b{y}_{\jbbRt{k}}}^2-\rho_{\jbbR{k-1}}^2} 
             = \frac{1}
          {\norm{\rb_{\jbbR{k-1}}}^2 \big(\frac{\norm{\b{y}_{\jbbRt{k}}}}{\norm{\rb_{\jbbR{k-1}}}}
           -\frac{\norm{\rb_{\jbbR{k-1}}}}{\norm{\biidx{p}{k-1}}}\big)
           \big(\frac{\norm{\b{y}_{\jbbRt{k}}}}{\norm{\rb_{\jbbR{k-1}}}}
           +\frac{\norm{\rb_{\jbbR{k-1}}}}{\norm{\biidx{p}{k-1}}}\big)},
\\ \beta_{\jbbR{k-1}} &= \frac{\rho_{\jbbR{k-1}}}{\theta_{k-1}} \gamma_{\jbbR{k-1}}
            = \frac{\norm{\rb_{\jbbR{k-1}}}^4}
  {(\norm{\biidx{p}{k-1}}\norm{\biidx{y}{\jbbR{k}}} - \norm{\biidx{r}{\jbbR{k-1}}}^2)
   (\norm{\biidx{p}{k-1}}\norm{\biidx{y}{\jbbR{k}}} + \norm{\biidx{r}{\jbbR{k-1}}}^2)}.
	\end{align*}
   Therefore we specify 
   $\p_{\jbbRt{k}} = \beta_{\jbbR{k-1}} \biidx{p}{k-1} + \gamma_{\jbbR{k-1}} \b{y}_{\jbbRt{k}}$.
\end{proof}

\subsection{$\W$ not the identity}
\label{subsec:W}
\jbbR{Lifting the assumption from Section \ref{subsec:conv}, suppose that} \jbb{ $\W=(\b{B}\tp\b{B})\inv$. In this case, we may represent it in triangular factorized form as $ \W = \btxiidx{R}{W}\tp \btxiidx{R}{W}$
(because $\W$ is symmetric and nonsingular). Further, define the transformed quantities
\begin{equation*}
    \biidxex{x}{k-1}{W} := \btxiidx{R}{W}\tpi \x_{\jbbR{k-1}}, \quad \biidxex{p}{k}{W} := \btxiidx{R}{W}\tpi \p_{\jbbRt{k}},
    \quad \text{ and } \quad \btxiidx{A}{W} := \A\btxiidx{R}{W}\tp.
\end{equation*}
Then update \eqref{eq:pkLong} becomes 
\begin{equation*}
    \biidxex{p}{k}{W} = \btxiidx{A}{W}\tp \Sk (\Sk\tp\btxiidx{A}{W} \btxiidx{A}{W}\tp \Sk )^{-1}
    \Sk\tp \rb_{\jbbR{k-1}}.
\end{equation*}
As for $ \p_{\jbbRt{k}} $, there is a recursion
\begin{equation}
    \label{eq:pkTransrec}
        \biidxex{p}{k}{W} = \iidxex{\beta}{k-1}{W}\biidxex{p}{k-1}{W} + \iidxex{\gamma}{k-1}{W}\biidxex{y}{k}{W}
\end{equation}
with certain scalars $\iidxex{\beta}{k-1}{W}$ and $\iidxex{\gamma}{k-1}{W}$, where
$\biidxex{y}{k}{W} := \btxiidx{A}{W}\tp \rb_{k-1}$. 
Because $ \biidxex{p}{k-1}{W} = \btxiidx{R}{W}\tpi \jbbRt{\p_{k}}  $, we find from \eqref{eq:pkTransrec},
by rewriting variables in terms of $ \jbbRt{\p_{k}} $ and $ \b{y}_{k} $, that
\begin{equation}
    \label{eq:pkShortW}
    \jbbRt{\p_{k}} = \iidxex{\beta}{k-1}{W} \biidx{p}{k-1} + \iidxex{\gamma}{k-1}{W} \W \b{y}_{k},
\end{equation}
where 
\begin{equation*}
    \begin{aligned}
    \iidxex{\beta}{k-1}{W} &= \frac{1}{ \bigg(\sqrt{\frac{(\biidx{p}{k-1}\tp \W\inv\biidx{p}{k-1})
    (\biidx{y}{k}\tp \W\biidx{y}{k})}{\norm{\rb_{k-1}}^2\norm{\rb_{k-1}}^2}}
    -1\bigg)
    \bigg(\sqrt{\frac{(\biidx{p}{k-1}\tp \W\inv\biidx{p}{k-1})
    (\biidx{y}{k}\tp \W\biidx{y}{k})}{\norm{\rb_{k-1}}^2\norm{\rb_{k-1}}^2}}
    +1\bigg)},
\\ 
    \iidxex{\gamma}{k-1}{W} &= 
	\frac{(\biidx{y}{k}\tp \W\biidx{y}{k})\inv}
	{\bigg(1-
	\sqrt{\frac{\norm{\rb_{k-1}}^2\norm{\rb_{k-1}}^2}{(\biidx{p}{k-1}\tp \W\inv\biidx{p}{k-1})(\biidx{y}{k}\tp \W\biidx{y}{k})}}\bigg)
	\bigg(1+
	\sqrt{\frac{\norm{\rb_{k-1}}^2\norm{\rb_{k-1}}^2}{(\biidx{p}{k-1}\tp \W\inv\biidx{p}{k-1})(\biidx{y}{k}\tp \W\biidx{y}{k})}}\bigg)}.
	\end{aligned}
\end{equation*}
Note that \eqref{eq:pkShortW} can be evaluated efficiently as long as products
and solves with $ \W $ are inexpensive. Therefore, we typically use the identity or a diagonal matrix that
normalizes the columns of $ \A $, namely
$ \W = \text{diag}\big(\frac{1}{\norm{\A_{:,1}}},\dots,\frac{1}{\norm{\A_{:,n}}}\big) $.
}

\section{Algorithm}
\label{sec:alg}
\jbb{Because of the short recursive update formula, our method can be implemented efficiently as in Algorithm \ref{alg:PLSS}.
The parameter matrix $\B$, where $\B\tp\B=\W\inv$ and $(\B\tp\B)\inv=\W$, is optional.}

\begin{algorithm}
\caption{{\small PLSS} (Projected Linear Systems Solver)} 
\label{alg:PLSS}
\begin{algorithmic}[1]
\ENSURE
$\mathtt{A\in(m \times n),~x_0,~b} $, $\mathtt{0<maxIt}$, ($ \mathtt{Optional: B}\tp\mathtt{B \in(n \times n))} $
\IF{$ \mathtt{B}\tp\mathtt{B} \ne \mathtt{Empty} $}
 	 \STATE{$ \mathtt{WI=\mathtt{B}\tp\mathtt{B}};$} 
\ELSE
    \STATE{$ \mathtt{WI=I;}~\mathtt{ \%~Identity~size~(n \times n)}$ }
\ENDIF
\STATE{$ \mathtt{ \%~Initialization}$ }

\STATE{\begin{tabular}{l l l}
     $\mathtt{x_k=x_0;}$~& $\phantom{\mathtt{WI}}\mathtt{~r_k=b-A*x_k;}$~& $\mathtt{~y_k=A\tp*r_k;}$ \\
     $\mathtt{\rho_k=r_k\tp*r_k;}$~& $\mathtt{~WIy_k =WI \backslash y_k;}$~&
     $\mathtt{~\phi_k=y_k\tp*WIy_k;}$ \\
\end{tabular}}


\STATE{$\mathtt{p_k=(\rho_k / \phi_k).*WIy_k;}$}
\STATE{$\mathtt{\theta_k=p_k\tp(WI*p_k);}$}
\STATE{$\mathtt{x_k=x_k+p_k;}$}

\WHILE{$\mathtt{k<maxIt}$} 
    \STATE{$\mathtt{k=k+1;}$}
    \STATE{$\mathtt{r_k=r_k-A*p_k;}$}
    \STATE{$\mathtt{y_k=A\tp*r_k;}$}
    \STATE{$\mathtt{\rho_k=r_k\tp*r_k;}$}
    \STATE{$\mathtt{WIy_k =WI \backslash y_k;}$}
    \STATE{$\mathtt{\phi_k=y_k\tp*Wiy_k;}$}
    \STATE{$\mathtt{\beta_k= 1 / ((sqrt(\theta_k*\phi_k) / \rho_k -1)*(sqrt(\theta_k*\phi_k) / \rho_k+1));}$}
    \STATE{$\mathtt{\gamma_k=(\theta_k/\rho_k)*\beta_k;}$}
    \STATE{$\mathtt{p_k=\beta_k*p_k+\gamma_k*WIy_k;}$}
    \STATE{$\mathtt{\theta_k=p_k\tp(WI*p_k);}$}
    \STATE{$\mathtt{x_k=x_k+p_k;}$}
\ENDWHILE
\end{algorithmic}
\end{algorithm}

\jbb{We emphasize that Algorithm \ref{alg:PLSS} updates only four vectors
$\mathtt{r_k,~y_k}$, $\mathtt{p_k}$ and $\mathtt{x_k}$ (and uses one intermediate vector $\mathtt{WIy_k}$).
When the matrix $\mathtt{WI}$ is diagonal (which is typically the case), we always apply 
it element-wise to a vector (i.e., $\mathtt{WI~.*~p_k}$ or $\mathtt{WI~.\backslash~y_k}$).
The algorithm uses one multiplication with $\mathtt{A}$ and one with $\mathtt{A}\tp$ per iteration.
\jbrv{For reference, we note that LSQR's \cite[Table 1]{PaigeSaunders82} dominant
    number of multiplications for large sparse $ \A$ at each iteration is $O(3m+5n)$ 
    (assuming $ \A\b{v} $ and $ \A^{\tp} \b{u} $ cost $O(m)$ and $O(n)$ multiplications).
    When $ \texttt{WI} $ is the identity then Algorithm \ref{alg:PLSS} incurs $O(2m+5n)$ multiplications.}
For practical implementation one may wish to change the stopping condition in the loop.
For instance, with consistent linear systems (where $ \A\x=\bbold $ can be solved exactly), 
the condition $\mathtt{sqrt(\rho_k)/ norm(b) < \tau}$ for a tolerance $\mathtt{\tau>0}$ may be used
to stop the iterations.} 

\section{\jbb{Relation to Craig's method}}
\label{sec:craigs}
\jbb{The orthogonal residuals in our method are reminiscent of Craig's method \cite{Paige74,PaigeSaunders82}.
Indeed when $\W=\I$, the updates $\biidx{p}{k-1}$ from \eqref{eq:pkLong} or \eqref{eq:pkShort}
\jbbRo{with $\Sbold_k $ from \eqref{eq:S}} correspond to the updates in Craig's method. For background, Craig's method can be developed using the Golub-Kahan 
bidiagonalization procedure as described in \cite[Secs.~3 \& 7.2]{PaigeSaunders82}. 
With $\beta_1 \b{u}_1 = \b{b}$, $\beta_1 \equiv \norm{\b{b}}$, the bidiagonalization
generates orthonormal
matrices $\Uk$ and $\Vk$ in $\mathbb{R}^{n\times k}$ and a lower bidiagonal
matrix $\Lk \in \mathbb{R}^{k \times k}$ with diagonals $\alpha_1,\ldots,\alpha_k$
and subdiagonals $\beta_2,\ldots,\beta_k$. 
After $k$ iterations, the following relations hold:
\begin{align}
     \b{A}\Vk &= \Uk \Lk+\beta_{k+1}\bko{u}\ek\tp,  \label{eq:bidiag1}
\\  \b{A}\tp\Uk &= \Vk \Lk\tp.                      \label{eq:bidiag2} 
\end{align}
With $\zk \in \mathbb{R}^k$ defined by
$\Lk \zk=\beta_1\biidx{e}{1}$ and $\zeta_k$ denoting the $k^{\text{th}}$ element of $\zk$, the iterates in Craig's method are computed as 
\begin{align}
   \xk    &= \Vk\zk=\biidx{x}{k-1}+\zeta_k \bk{v},
\\ \bk{r} &= \b{b}-\b{A}\bk{x}=-\zeta_k\beta_{k+1}\bko{u}.  \label{eq:craigRes}
\end{align}
We prove that iterates generated by Craig's method are equivalent to \eqref{eq:pkLong} with $\b{W}=\b{I}$ 
\jbbRo{and using residuals in a sketch}.

\begin{theorem}
	\label{thrm:craigsMethod}
	The updates $\zeta_k\bk{v}\equiv\biidx{p}{k}^{\textnormal{C}}$ in Craig's method are equal to the 
	updates $\biidx{p}{k}$ in \eqref{eq:pkLong} or \eqref{eq:pkShort} with $\W=\I$ \jbbRo{and $\Sbold_k$ in \eqref{eq:S}}.
\end{theorem}

\begin{proof}
First note that the residual in Craig's method is 
\begin{equation*}
    \bk{r}=\b{b}-\b{A}\bk{x}=\b{b}-\b{A}(\biidx{x}{k-1}+\zeta_k\bk{v})=\biidx{r}{k-1}-\zeta_k\b{A}\bk{v},
\end{equation*}
and also $\bk{r}=-\zeta_k\beta_{k+1}\bko{u}$ from \eqref{eq:craigRes}. By orthogonality of $\bk{U}$ it holds that
\begin{equation}
    \label{eq:orthRes}
    \bk{U}\tp\bk{r}=\b{0} \quad \text{ and } \quad \bk{U}\tp\b{A}(\zeta_k\bk{v})=\bk{U}\tp\biidx{r}{k-1}.
\end{equation}
We define the update as $\zeta_k\bk{v}\equiv\biidx{p}{k}^{\textnormal{C}}$. The scalar $\zeta_k$
is defined recursively as $\zeta_k = -\frac{\beta_k}{\alpha_k} \zeta_{k-1} $ (with $\zeta_0\equiv-1$, cf. \cite[Sec.~7.2]{PaigeSaunders82}). From \eqref{eq:bidiag1}--\eqref{eq:bidiag2}
 we deduce
 \begin{align*}
     \b{A}\tp\biidx{r}{k-1} &= \b{A}\tp(-\zeta_{k-1}\beta_k\bk{u}) \\
     &= \b{A}\tp(-\zeta_{k-1}\beta_k\bk{U}\bk{e}) \\
     &= -\zeta_{k-1}\beta_k\bk{V}\bk{L}\tp \bk{e} \\
     &= -\zeta_{k-1}\beta_k\bk{V}\bmat{\b{0} \\ \beta_k \\ \alpha_k} \\
     &= -\zeta_{k-1}\beta_k \alpha_k \bk{v} -\zeta_{k-1}\beta_k\bk{V}\bmat{\b{0} \\ \beta_k \\ 0} \\
     &= \alpha_k^2 \biidx{p}{k}^{\textnormal{C}} -\zeta_{k-1}\beta_k\bk{V}(\bk{L}\tp-\bk{L})\bk{e}.
 \end{align*}
 Using 
 $\phi_k \equiv -\alpha_k^2/(\zeta_{k-1}\beta_k) $ and a vector $\bs{\lambda}_{k} \in \mathbb{R}^{k}$, we have
 \begin{equation}
    \label{eq:craigTwo}
     \frac{-\b{A}\tp\biidx{r}{k-1}}{\zeta_{k-1}\beta_k}  =
     \b{A}\tp\bk{u}=\frac{-\alpha_k^2}{\zeta_{k-1}\beta_k} \biidx{p}{k}^{\textnormal{C}}+
     \bk{V}(\bk{L}\tp-\bk{L})\bk{e}\equiv
     \phi_k \biidx{p}{k}^{\textnormal{C}} + \b{A}\tp \bk{U}\bs{\lambda}_{k}.
 \end{equation}
 Combining the second equality in \eqref{eq:orthRes} with \eqref{eq:craigTwo} gives the system
 \begin{equation*}
     \bmat{\phi_k \I & \b{A}\tp \bk{U} \\ \bk{U}\tp \b{A} & \b{0}}
     \bmat{\biidx{p}{k}^{\textnormal{C}} \\ \bs{\lambda}_{k}} = 
     \bmat{ \b{A}\tp \bk{u} \\ \bk{U}\tp \biidx{r}{k-1} },
 \end{equation*}
 which leads to
 \begin{equation}
    \label{eq:craigStep}
     \b{A}\tp\bk{U}(\bk{U}\tp\b{A}\b{A}\tp\bk{U})^{-1}\bk{U}\tp\biidx{r}{k-1}=\biidx{p}{k}^{\textnormal{C}}.
 \end{equation}
 With $\biidx{S}{\jbbR{k}}$ as defined in \eqref{eq:S}, matrices $ \bk{U} $ and $ \biidx{S}{\jbbR{k}} $
 are closely related:
 \begin{equation*}
     \bk{U} = \bmat{\biidx{r}{0} & \biidx{r}{1} & \cdots & \biidx{r}{k-1}} \bk{D}=\biidx{S}{k}\bk{D}, \quad \bk{D} \equiv -\text{diag}(\zeta_0\beta_1,\zeta_1\beta_2,
     \ldots,\zeta_{k-1}\beta_k).
 \end{equation*}
 Substituting $ \bk{U}=\biidx{S}{\jbbR{k}}\bk{D} $ into \eqref{eq:craigStep}, we obtain
 \begin{equation}
    \label{eq:craigStepS}
     \b{A}\tp\biidx{S}{\jbbR{k}}(\biidx{S}{\jbbR{k}}\tp\b{A}\b{A}\tp\biidx{S}{\jbbR{k}})^{-1}\biidx{S}{\jbbR{k}}\tp\biidx{r}{k-1}=\biidx{p}{k}^{\textnormal{C}}.
 \end{equation}
 \jbbR{Comparing \eqref{eq:craigStepS} with \eqref{eq:pkLong} we see that the updates are the same
 when $\W=\I$.} 
 \end{proof}}

\section{\jbbRo{PLSS Kaczmarz}}
\label{sec:kaczmark}
\jbbRo{Another way to construct the sketching matrix is to 
concatenate columns of the identity matrix $\ek$ in the sketch $\Sk = \bmat{\biidx{e}{1} & \biidx{e}{2} & \cdots & \biidx{e}{k}}$. When the columns are assembled in a random
order, this process generalizes the randomized Kaczmarz iteration \cite{Kaczmarz,StrohmerVershynin09}. We denote
a random identity column by $\b{e}_{i_k}$. From $\Sbold_{k-1}\tp \b{r}_{k-1} = \b{0}$ we have $ \Sk\tp \biidx{r}{k-1} = r_{i_k}^{(k-1)} \b{e}_{i_k} \tp $, where $r_{i_k}^{(k-1)}$ is the
${i_k}^{\text{th}}$ element of $\biidx{r}{k-1}$. Thus the updates with this sketch are (cf.~\eqref{eq:pkLong})
\begin{equation*}
    \biidx{p}{k} = r_{i_k}^{(k-1)} \Yk (\Yk\tp \Yk)^{-1} \b{e}_{i_k}. 
\end{equation*}
In the notation of Sections \ref{subsec:orthP} and \ref{subsec:lincomb}, the updates are orthogonal and can be 
represented  by
\begin{equation*}
    \jbbRt{\p_{k}} = \b{P}_{k-1} \b{g}_{k-1} + \gamma_{k-1} \jbbRt{\y_{k}}.
\end{equation*}
Since $ \biidx{y}{k} = \b{A}\tp \b{e}_{i_k} = \biidx{a}{i_k} $ (the ${i_k}^{\textnormal{th}}$ row of $ \b{A} $) and 
from $ \jbbRt{\y_{k}\tp}\jbbRt{\p_{k}} = r_{i_k}^{(k-1)} = \b{e}_{i_k} \tp (\b{b} -\A \x_{k-1}) $, we develop the update
\begin{align}
    \label{eq:pkKacz}
    \biidx{p}{k} &= \gamma_{k-1}(\biidx{a}{i_k} - \biidx{P}{k-1} \bs{\Theta}^{-1}_{k-1} \biidx{P}{k-1}\tp \biidx{a}{i_k} ), \\
    \gamma_{k-1} &= \frac{\b{e}_{i_k} \tp (\b{b} - \b{A} \x_{k-1})}{(\| \biidx{a}{i_k} \|_2 - \|\biidx{d}{k-1} \|_2)(\| \biidx{a}{i_k} \|_2 + \|\biidx{d}{k-1} \|_2) },\\ 
    \biidx{d}{k-1} &\equiv \bs{\Theta}^{-1/2}_{k-1} \biidx{P}{k-1} \tp\biidx{a}{i_k}. 
\end{align}
If the history of previous updates is not used, i.e., $ \biidx{P}{k-1} = \b{0} $,
then $ \biidx{p}{k} $ is equal to the update in the Kaczmarz method. The Kaczmarz method is useful,
especially in situations where the entire matrix is not accessible (possibly because of its size), because it enables versions that access only one row of $ \b{A} $ each iteration. Note that update \eqref{eq:pkKacz} can be implemented with one multiplication of 
$ \biidx{P}{k-1}  $ and one $ \b{A} $. To compute the next update using 
$ \biidx{r}{k} = \biidx{r}{k-1} - \b{A} \biidx{p}{k} $ and 
$\ek \tp \biidx{P}{k} \tp \biidx{a}{i_k}\tp = \biidx{p}{k} \tp \b{A}\tp \b{e}_{i_k} $,
we compute $ \b{A}\biidx{p}{k} $ 
and append it to an array that stores
$ \bmat{ \b{A} \biidx{p}{1} & \cdots \b{A} \biidx{p}{k} } \tp = \biidx{P}{k}\tp \b{A}\tp $.
With this, $ \biidx{p}{k+1}$ can be computed by selecting the $(k+1)^{\textnormal{th}}$ column of $ \biidx{P}{k}\tp \b{A}\tp $ to obtain $\biidx{P}{k}\tp \biidx{a}{i_{k+1}}\tp$, then 
multiplying $ \biidx{P}{k} ( \bs{\Theta}^{-1/2}_{k} \bk{d} ) $ and forming
\jbbRt{$\p_{k+1} = \gamma_k (\biidx{a}{i_{k+1}}\tp - \biidx{P}{k} ( \bs{\Theta}^{-1/2}_{k} \bk{d} ) ) $}.}

\section{Numerical experiments}
\label{sec:numex}

\jbbRo{Our algorithms are implemented in MATLAB and PYTHON 3.9.} \jbbR{The numerical experiments are carried out in MATLAB 2016a on a 
MacBook Pro @2.6 GHz Intel Core i7 with 32 GB of memory. For comparisons, we use the implementations of \cite{GowerRichtarik15},
a randomized version of \eqref{eq:pkLong}, Algorithm \ref{alg:PLSS}, {\small LSQR}, {\small \jbrv{LSMR}} \cite{FongSaunders11} and
{\small CRAIG} \cite[p58]{PaigeSaunders82}, \cite{SaundersCRAIG}.
All codes are available in the public domain \cite{PLSSsoftware}.
The stopping criterion is either $ \|\b{A}\bk{x} - \bbold \|_2/ \|\bbold\|_2 \le \epsilon $ or $ \|\b{A}\bk{x} - \bbold\|_2 \le \epsilon $, depending on the experiment.
Unless otherwise specified, the iteration limit is $n$. We label Algorithm~\ref{alg:PLSS} 
with $\b{W}=\b{I}$ and $\b{W}=\text{diag}(\frac{1}{\norm{\A_{:,1}}},\cdots,\frac{1}{\norm{\A_{:,n}}})$ as {\small PLSS} and {\small PLSS W} respectively.} \jbbRo{Our implementation of 
update \eqref{eq:pkKacz} with random columns of the identity matrix is called {\small PLSS KZ }.}

\subsection{Experiment I}
\label{sec:numex1}

\jbbR{This experiment uses moderately large sparse matrices with
$m > n$ and $10^3 \le n \le 10^4$. All linear systems are consistent and we set
$\texttt{x=ones(n,1); x(1)=10; b=A*x;}$ and $\x_0 = \b{0}$. \jbrv{The condition number
of each matrix is in Table \ref{tab:cond}.} For reference,
we add the method ``{Rand.~Proj.}'', which uses update formula \eqref{eq:pkLong} (with $\b{W}=\I$) with $\Sk \in \mathbb{R}^{m \times r}$ being a random standard normal matrix. The sketching 
parameter is $r=10$. Each iteration with this method thus recomputes $\Yk=\A\tp\Sk \in \mathbb{R}^{n \times r}$,
$\Yk\tp \Yk$, and a solve with the latter matrix. Thus the computational effort with this approach is typically
much larger than with the proposed methods. Relative residuals $\|\A\xk-\bbold\|_2 /\|\bbold\|_2 \le \epsilon$ are used to stop
with $\epsilon=10^{-2}$. Table~\ref{tab:exp1} records detailed outcomes of the experiment.}

\begin{table}[t]  
 \captionsetup{width=1\linewidth}
\caption{\jbbR{Experiment I compares 4 solvers on 42 linear systems from the SuiteSparse Matrix Collection \cite{SuiteSparseMatrix} with stopping tolerance $\epsilon = 10^{-2}$ and iteration limit $n$. 
\jbrv{In column 3, ``Rank'' is the structural rank of the matrix and} ``Dty"  is the density of a particular matrix $ \b{A} $ calculated as $ \text{Dty} = \frac{\text{nnz}(\b{A})}{m \cdot n} $.
Entries with 
superscript $^\dagger$ 
denote problems for which the solver did not converge to the specified tolerance. 
Bold entries mark the
fastest times, while second fastest times are italicized.
}}
\label{tab:exp1}

\setlength{\tabcolsep}{2pt} 

\scriptsize
\hbox to 1.00\textwidth{\hss 
\begin{tabular}{|l | c | c | c c c | c c c | c c c | c c c |} 
		\hline
		\multirow{2}{*}{Problem} & \multirow{2}{*}{$m$/$n$} & \multirow{2}{*}{\text{\jbrv{Rank}}/\text{Dty}} & 
        \multicolumn{3}{c|}{\jbbo{Rand. Proj.}}  & 
		\multicolumn{3}{c|}{\jbbo{PLSS}} &  		
		\multicolumn{3}{c|}{\jbbo{PLSS W}} &  	
		\multicolumn{3}{c|}{\jbbo{LSQR}} 	\\
		\cline{4-15}
		& & 
		& It &  Sec & Res
		& It &  Sec & Res
		& It &  Sec & Res
		& It &  Sec & Res \\
		\hline
$\texttt{lpi\_gran}$ & 2658/2525 & 2311/0.003 &1192 &0.8& 0.01 &62 &\emph{0.012}& 0.01 &26 &\textbf{0.011}& 0.01 &2525 &0.42& 3e-05 \\ 
$\texttt{landmark}$ & 71952/2704 & 2673/0.006 &$\textnormal{2704}^{\dagger}$ &$\textnormal{46}^{\dagger}$& $ \textnormal{0.02}^{\dagger}$ &33 &\emph{0.083}& 0.01 &9 &\textbf{0.039}& 0.008 &1196 &2.8& 4e-05 \\ 
$\texttt{Kemelmacher}$ & 28452/9693 & 9693/0.0004 &$\textnormal{9693}^{\dagger}$ &$\textnormal{51}^{\dagger}$& $ \textnormal{0.7}^{\dagger}$ &2258 &0.96& 0.01 &1323 &\emph{0.63}& 0.01 &839 &\textbf{0.52}& 0.005 \\ 
$\texttt{Maragal\_4}$ & 1964/1034 & 995/0.01 &$\textnormal{1034}^{\dagger}$ &$\textnormal{0.56}^{\dagger}$& $ \textnormal{0.08}^{\dagger}$ &97 &\emph{0.0085}& 0.009 &57 &\textbf{0.0063}& 0.009 &362 &0.045& 0.0002 \\ 
$\texttt{Maragal\_5}$ & 4654/3320 & 2690/0.006 &$\textnormal{3320}^{\dagger}$ &$\textnormal{4.8}^{\dagger}$& $ \textnormal{0.08}^{\dagger}$ &181 &\emph{0.052}& 0.01 &141 &\textbf{0.047}& 0.009 &588 &0.2& 0.0001 \\ 
$\texttt{Franz4}$ & 6784/5252 & 5252/0.001 &3878 &5.7& 0.01 &8 &\textbf{0.0017}& 0.007 &4 &\emph{0.0019}& 0.003 &10 &0.0044& 7e-05 \\ 
$\texttt{Franz5}$ & 7382/2882 & 2882/0.002 &2536 &3.9& 0.01 &4 &\textbf{0.00081}& 0.001 &3 &\emph{0.0014}& 0.005 &5 &0.0019& 0.0001 \\ 
$\texttt{Franz6}$ & 7576/3016 & 3016/0.002 &2434 &3.8& 0.01 &3 &\textbf{0.00069}& 0.006 &4 &\emph{0.0016}& 0.003 &5 &0.0017& 5e-05 \\ 
$\texttt{Franz7}$ & 10164/1740 & 1740/0.002 &1555 &3& 0.01 &1 &\textbf{0.00042}& 0 &2 &\emph{0.00087}& 0.005 &1 &0.00088& 3e-15 \\ 
$\texttt{Franz8}$ & 16728/7176 & 7176/0.0008 &6541 &22& 0.01 &4 &\textbf{0.0015}& 0.005 &4 &\emph{0.0032}& 0.004 &7 &0.0033& 0.0001 \\ 
$\texttt{Franz9}$ & 19588/4164 & 4164/0.001 &3270 &13& 0.01 &7 &\emph{0.0034}& 0.006 &4 &\textbf{0.0031}& 0.005 &12 &0.0076& 2e-06 \\ 
$\texttt{Franz10}$ & 19588/4164 & 4164/0.001 &3322 &13& 0.01 &7 &\emph{0.003}& 0.006 &4 &\textbf{0.0028}& 0.005 &12 &0.0072& 2e-06 \\ 
$\texttt{GL7d12}$ & 8899/1019 & 1019/0.004 &$\textnormal{1019}^{\dagger}$ &$\textnormal{1.8}^{\dagger}$& $ \textnormal{0.03}^{\dagger}$ &14 &\emph{0.003}& 0.008 &7 &\textbf{0.0019}& 0.008 &32 &0.0085& 4e-05 \\ 
$\texttt{GL7d13}$ & 47271/8899 & 8897/0.0008 &$\textnormal{8899}^{\dagger}$ &$\textnormal{1.1e+02}^{\dagger}$& $ \textnormal{0.02}^{\dagger}$ &19 &\emph{0.033}& 0.009 &7 &\textbf{0.017}& 0.009 &44 &0.08& 1e-05 \\ 
$\texttt{ch6-6-b3}$ & 5400/2400 & 2400/0.002 &2035 &2.3& 0.01 &4 &\textbf{0.0005}& 0.008 &4 &\emph{0.00081}& 0.008 &6 &0.0016& 0.0003 \\ 
$\texttt{ch7-6-b3}$ & 12600/4200 & 4200/0.001 &3531 &8.6& 0.01 &4 &\textbf{0.0011}& 0.004 &4 &\emph{0.0018}& 0.004 &6 &0.0023& 0.0002 \\ 
$\texttt{ch7-8-b2}$ & 11760/1176 & 1176/0.003 &1003 &2.2& 0.01 &3 &0.002& 0.001 &3 &\textbf{0.001}& 0.001 &5 &\emph{0.0019}& 4e-06 \\ 
$\texttt{ch7-9-b2}$ & 17640/1512 & 1512/0.002 &1326 &4.2& 0.01 &3 &\textbf{0.001}& 0.001 &3 &\emph{0.0014}& 0.001 &5 &0.0024& 6e-06 \\ 
$\texttt{ch8-8-b2}$ & 18816/1568 & 1568/0.002 &1377 &4.7& 0.01 &3 &\textbf{0.0011}& 0.0005 &3 &\emph{0.0014}& 0.0005 &4 &0.0022& 7e-16 \\ 
$\texttt{cis-n4c6-b3}$ & 5970/1330 & 1330/0.003 &1118 &1.3& 0.01 &2 &\textbf{0.00024}& 2e-17 &2 &\emph{0.00049}& 2e-17 &2 &0.00082& 5e-16 \\ 
$\texttt{cis-n4c6-b4}$ & 20058/5970 & 5970/0.0008 &4462 &17& 0.01 &2 &\textbf{0.0011}& 0.005 &2 &\emph{0.002}& 0.008 &4 &0.0034& 7e-05 \\ 
$\texttt{mk10-b3}$ & 4725/3150 & 3150/0.001 &3047 &2.9& 0.01 &6 &\textbf{0.00065}& 0.002 &6 &\emph{0.0011}& 0.002 &7 &0.0017& 7e-16 \\ 
$\texttt{mk11-b3}$ & 17325/6930 & 6930/0.0006 &6079 &20& 0.01 &5 &\textbf{0.0016}& 0.001 &5 &\emph{0.0023}& 0.001 &6 &0.0032& 4e-16 \\ 
$\texttt{mk12-b2}$ & 13860/1485 & 1485/0.002 &1306 &3.3& 0.01 &3 &\textbf{0.00086}& 0.002 &3 &\emph{0.0011}& 0.002 &4 &0.0019& 8e-16 \\ 
$\texttt{n2c6-b4}$ & 3003/1365 & 1365/0.004 &929 &0.58& 0.01 &1 &\textbf{0.00016}& 8e-17 &1 &\emph{0.00031}& 2e-16 &1 &0.00059& 4e-14 \\ 
$\texttt{n2c6-b5}$ & 4945/3003 & 3003/0.002 &1861 &1.9& 0.01 &1 &\textbf{0.00022}& 2e-16 &1 &\emph{0.00049}& 2e-16 &1 &0.0016& 4e-16 \\ 
$\texttt{n2c6-b6}$ & 5715/4945 & 4945/0.001 &3585 &4.9& 0.01 &4 &\textbf{0.0011}& 0.008 &4 &\emph{0.002}& 0.008 &8 &0.0037& 0.0001 \\ 
$\texttt{n3c6-b4}$ & 3003/1365 & 1365/0.004 &907 &0.56& 0.01 &1 &\textbf{0.00016}& 8e-17 &1 &\emph{0.00029}& 2e-16 &1 &0.00059& 4e-14 \\ 
$\texttt{n3c6-b5}$ & 5005/3003 & 3003/0.002 &1825 &1.9& 0.01 &1 &\textbf{0.00021}& 2e-16 &1 &\emph{0.00048}& 2e-16 &1 &0.00066& 4e-16 \\ 
$\texttt{n3c6-b6}$ & 6435/5005 & 5005/0.001 &2729 &3.8& 0.01 &1 &\textbf{0.00035}& 8e-17 &1 &0.0013& 2e-16 &1 &\emph{0.00098}& 1e-13 \\ 
$\texttt{n4c5-b4}$ & 2852/1350 & 1350/0.004 &953 &0.57& 0.01 &3 &\textbf{0.0003}& 0.004 &3 &\emph{0.00045}& 0.003 &5 &0.0011& 0.0001 \\ 
$\texttt{n4c5-b5}$ & 4340/2852 & 2852/0.002 &1925 &1.8& 0.01 &3 &\textbf{0.00044}& 0.009 &3 &\emph{0.00075}& 0.005 &5 &0.0014& 0.0003 \\ 
$\texttt{n4c5-b6}$ & 4735/4340 & 4340/0.002 &2417 &2.7& 0.01 &3 &\textbf{0.00069}& 0.008 &3 &\emph{0.0013}& 0.007 &7 &0.0029& 8e-05 \\ 
$\texttt{n4c6-b3}$ & 5970/1330 & 1330/0.003 &1114 &1.3& 0.01 &2 &\textbf{0.00026}& 2e-17 &2 &\emph{0.00043}& 2e-17 &2 &0.00072& 5e-16 \\ 
$\texttt{n4c6-b4}$ & 20058/5970 & 5970/0.0008 &4483 &17& 0.01 &2 &\textbf{0.0012}& 0.005 &2 &\emph{0.0018}& 0.008 &4 &0.0032& 7e-05 \\ 
$\texttt{rel7}$ & 21924/1045 & 1043/0.002 &0 &0.0047& 0 &0 &\textbf{0.00025}& 0 &0 &\emph{0.00085}& 0 &0 &0.0014& 0 \\ 
$\texttt{relat7b}$ & 21924/1045 & 1043/0.004 &0 &0.0042& 0 &0 &\textbf{0.00021}& 0 &0 &0.00082& 0 &0 &\emph{0.00043}& 0 \\ 
$\texttt{relat7}$ & 21924/1045 & 1043/0.004 &0 &0.004& 0 &0 &\textbf{0.00022}& 0 &0 &0.00058& 0 &0 &\emph{0.00041}& 0 \\ 
$\texttt{mesh\_deform}$ & 234023/9393 & 9393/0.0004 &$\textnormal{9393}^{\dagger}$ &$\textnormal{4.2e+02}^{\dagger}$& $ \textnormal{0.1}^{\dagger}$ &290 &\emph{0.95}& 0.01 &243 &\textbf{0.82}& 0.01 &551 &1.9& 0.0001 \\ 
$\texttt{162bit}$ & 3606/3597 & 3460/0.003 &$\textnormal{3597}^{\dagger}$ &$\textnormal{3.7}^{\dagger}$& $ \textnormal{0.02}^{\dagger}$ &90 &\emph{0.017}& 0.01 &44 &\textbf{0.01}& 0.01 &1008 &0.3& 1e-05 \\ 
$\texttt{176bit}$ & 7441/7431 & 7110/0.001 &$\textnormal{7431}^{\dagger}$ &$\textnormal{16}^{\dagger}$& $ \textnormal{0.02}^{\dagger}$ &141 &\emph{0.055}& 0.009 &46 &\textbf{0.021}& 0.01 &1688 &0.89& 7e-06 \\ 
$\texttt{specular}$ & 477976/1600 & 1442/0.01 &$\textnormal{1600}^{\dagger}$ &$\textnormal{2.2e+02}^{\dagger}$& $ \textnormal{0.02}^{\dagger}$ &50 &\emph{0.81}& 0.01 &6 &\textbf{0.2}& 0.009 &1600 &26& 3e-05 \\		
 \hline
 \end{tabular}
 \hss}
\end{table} 

    \subsection{\jbrv{Experiment II}}
    \label{sec:EX_II_new}
    \jbrv{In order to place our algorithms among state-of-the-art solvers, we 
            rerun the problems from Experiment I using the convergence
            criterion $\|\A\xk-\bbold\|_2 / \| \bbold \|_2 \le \epsilon $ with $\epsilon=10^{-6}$. Note that 
            this criterion is stricter than the one from before. The maximum number of iterations is
            $n+1000$.
            Instead of the randomized projection, we include {\small LSMR} as an additional solver. 
            Overall, we observe in Table~\ref{tab:exp5} that the four algorithms are competitive with each other in terms of
            iterations and times. Additionally, {\small PLSS}, {\small LSQR} and {\small LSMR} converged 
            on all but 3 problems, while {\small PLSS W} converged on all problems but one.}
    
    
    \begin{table}[p]  
 \captionsetup{width=1\linewidth}
\caption{\jbbR{Experiment II compares 4 solvers on 42 linear systems from the SuiteSparse Matrix Collection \cite{SuiteSparseMatrix} with stopping tolerance $\epsilon = 10^{-6}$ and iteration limit $n+1000$. 
\jbrv{In column 3, ``Rank'' is the structural rank of the matrix and} ``Dty" is the density of a particular matrix $ \b{A} $ calculated as $ \text{Dty} = \frac{\text{nnz}(\b{A})}{m \cdot n} $.
Entries with 
$^{\dagger}$ 
denote problems for which the solver did not converge to the specified tolerance. 
Bold entries mark the
fastest times, while second fastest times are italicized.
}}
\label{tab:exp5}

\setlength{\tabcolsep}{2pt} 

\scriptsize
\hbox to 1.00\textwidth{\hss 
\begin{tabular}{|l | c | c | c c c | c c c | c c c | c c c |} 
		\hline
		\multirow{2}{*}{Problem} & \multirow{2}{*}{$m$/$n$} & \multirow{2}{*}{\text{\jbrv{Rank}}/\text{Dty}} & 
        \multicolumn{3}{c|}{\jbbo{PLSS}}  & 
		\multicolumn{3}{c|}{\jbbo{PLSS W}} &  		
		\multicolumn{3}{c|}{\jbbo{LSQR}} &  	
		\multicolumn{3}{c|}{\jbrv{LSMR}} 	\\
		\cline{4-15}
		& & 
		& It &  Sec & Res
		& It &  Sec & Res
		& It &  Sec & Res
		& It &  Sec & Res \\
		\hline
$\texttt{lpi\_gran}$ & 2658/2525 & 2311/0.003 &$\textnormal{3525}^{\dagger}$ &$\textnormal{\emph{0.37}}^{\dagger}$& $ \textnormal{0.0002}^{\dagger}$ &$\textnormal{3525}^{\dagger}$ &$\textnormal{0.38}^{\dagger}$& $ \textnormal{0.0001}^{\dagger}$ &$\textnormal{3525}^{\dagger}$ &$\textnormal{0.54}^{\dagger}$& $ \textnormal{1e-05}^{\dagger}$ &$\textnormal{3525}^{\dagger}$ &$\textnormal{\textbf{0.34}}^{\dagger}$& $ \textnormal{2e-05}^{\dagger}$ \\ 
$\texttt{landmark}$ & 71952/2704 & 2673/0.006 &$\textnormal{3704}^{\dagger}$ &$\textnormal{9}^{\dagger}$& $ \textnormal{0.0002}^{\dagger}$ &153 &\textbf{0.39}& 9e-07 &$\textnormal{3704}^{\dagger}$ &$\textnormal{8.7}^{\dagger}$& $ \textnormal{7e-06}^{\dagger}$ &$\textnormal{3704}^{\dagger}$ &$\textnormal{\emph{8.2}}^{\dagger}$& $ \textnormal{9e-06}^{\dagger}$ \\ 
$\texttt{Kemelmacher}$ & 28452/9693 & 9693/0.0004 &3199 &1.4& 1e-06 &1824 &\textbf{0.9}& 9e-07 &3010 &1.8& 1e-06 &3112 &\emph{1.4}& 1e-06 \\ 
$\texttt{Maragal\_4}$ & 1964/1034 & 995/0.01 &838 &0.069& 1e-06 &363 &\textbf{0.032}& 9e-07 &688 &0.083& 1e-06 &715 &\emph{0.06}& 1e-06 \\ 
$\texttt{Maragal\_5}$ & 4654/3320 & 2690/0.006 &4177 &1.2& 1e-06 &1655 &\textbf{0.49}& 7e-07 &3934 &1.3& 1e-06 &4121 &\emph{1}& 1e-06 \\ 
$\texttt{Franz4}$ & 6784/5252 & 5252/0.001 &12 &0.0032& 2e-14 &9 &\emph{0.0026}& 1e-09 &11 &0.0047& 4e-13 &11 &\textbf{0.0024}& 3e-13 \\ 
$\texttt{Franz5}$ & 7382/2882 & 2882/0.002 &8 &\emph{0.0016}& 2e-15 &9 &0.0025& 2e-08 &7 &0.0022& 3e-07 &7 &\textbf{0.0011}& 3e-07 \\ 
$\texttt{Franz6}$ & 7576/3016 & 3016/0.002 &7 &\emph{0.0014}& 2e-16 &10 &0.0028& 1e-09 &6 &0.0021& 9e-17 &6 &\textbf{0.00097}& 5e-16 \\ 
$\texttt{Franz7}$ & 10164/1740 & 1740/0.002 &2 &\emph{0.0004}& 4e-16 &4 &0.00088& 5e-17 &1 &0.00075& 3e-15 &1 &\textbf{0.00031}& 3e-15 \\ 
$\texttt{Franz8}$ & 16728/7176 & 7176/0.0008 &12 &\emph{0.0036}& 7e-08 &12 &0.0056& 6e-08 &11 &0.0048& 4e-07 &11 &\textbf{0.0035}& 4e-07 \\ 
$\texttt{Franz9}$ & 19588/4164 & 4164/0.001 &14 &0.0055& 6e-13 &10 &\emph{0.0054}& 2e-08 &13 &0.008& 7e-07 &13 &\textbf{0.0052}& 7e-07 \\ 
$\texttt{Franz10}$ & 19588/4164 & 4164/0.001 &14 &0.0063& 6e-13 &10 &\emph{0.0059}& 2e-08 &13 &0.0083& 7e-07 &13 &\textbf{0.005}& 7e-07 \\ 
$\texttt{GL7d12}$ & 8899/1019 & 1019/0.004 &48 &0.0094& 6e-07 &22 &\textbf{0.0047}& 3e-07 &45 &0.011& 9e-07 &46 &\emph{0.0087}& 9e-07 \\ 
$\texttt{GL7d13}$ & 47271/8899 & 8897/0.0008 &57 &0.097& 7e-07 &30 &\textbf{0.056}& 5e-07 &55 &0.1& 8e-07 &55 &\emph{0.093}& 1e-06 \\ 
$\texttt{ch6-6-b3}$ & 5400/2400 & 2400/0.002 &10 &\textbf{0.0011}& 2e-17 &10 &0.0013& 2e-17 &9 &0.0023& 2e-16 &9 &\emph{0.0011}& 3e-17 \\ 
$\texttt{ch7-6-b3}$ & 12600/4200 & 4200/0.001 &11 &\emph{0.0029}& 3e-08 &11 &0.0038& 3e-08 &10 &0.0035& 2e-07 &10 &\textbf{0.0025}& 2e-07 \\ 
$\texttt{ch7-8-b2}$ & 11760/1176 & 1176/0.003 &7 &\textbf{0.0013}& 1e-08 &7 &0.0016& 1e-08 &6 &0.0022& 3e-07 &6 &\emph{0.0014}& 3e-07 \\ 
$\texttt{ch7-9-b2}$ & 17640/1512 & 1512/0.002 &7 &\emph{0.0019}& 5e-17 &7 &0.0026& 1e-16 &6 &0.0027& 3e-07 &6 &\textbf{0.0018}& 3e-07 \\ 
$\texttt{ch8-8-b2}$ & 18816/1568 & 1568/0.002 &5 &\emph{0.0014}& 8e-16 &5 &0.0021& 7e-16 &4 &0.0025& 7e-16 &4 &\textbf{0.0012}& 6e-16 \\ 
$\texttt{cis-n4c6-b3}$ & 5970/1330 & 1330/0.003 &3 &\emph{0.0004}& 1e-16 &3 &0.00044& 2e-17 &2 &0.00075& 5e-16 &2 &\textbf{0.00027}& 5e-16 \\ 
$\texttt{cis-n4c6-b4}$ & 20058/5970 & 5970/0.0008 &8 &\emph{0.0032}& 8e-09 &8 &0.0049& 6e-08 &7 &0.0049& 3e-07 &7 &\textbf{0.0031}& 3e-07 \\ 
$\texttt{mk10-b3}$ & 4725/3150 & 3150/0.001 &8 &\emph{0.0011}& 1e-16 &8 &0.0015& 2e-16 &7 &0.0016& 7e-16 &7 &\textbf{0.0008}& 6e-16 \\ 
$\texttt{mk11-b3}$ & 17325/6930 & 6930/0.0006 &7 &\emph{0.0022}& 2e-16 &7 &0.0033& 2e-16 &6 &0.0031& 4e-16 &6 &\textbf{0.0022}& 4e-16 \\ 
$\texttt{mk12-b2}$ & 13860/1485 & 1485/0.002 &5 &\textbf{0.0011}& 4e-16 &5 &0.0013& 5e-16 &4 &0.0019& 8e-16 &4 &\emph{0.0011}& 1e-15 \\ 
$\texttt{n2c6-b4}$ & 3003/1365 & 1365/0.004 &2 &\textbf{0.00015}& 3e-14 &2 &0.00029& 3e-14 &1 &0.00053& 4e-14 &1 &\emph{0.00018}& 4e-14 \\ 
$\texttt{n2c6-b5}$ & 4945/3003 & 3003/0.002 &2 &\emph{0.00037}& 2e-16 &2 &0.00054& 2e-16 &1 &0.00058& 4e-16 &1 &\textbf{0.00023}& 2e-16 \\ 
$\texttt{n2c6-b6}$ & 5715/4945 & 4945/0.001 &14 &\emph{0.0034}& 3e-07 &14 &0.0041& 2e-07 &13 &0.0057& 7e-07 &13 &\textbf{0.003}& 8e-07 \\ 
$\texttt{n3c6-b4}$ & 3003/1365 & 1365/0.004 &2 &\textbf{0.00016}& 3e-14 &2 &0.00028& 3e-14 &1 &0.00054& 4e-14 &1 &\emph{0.00019}& 4e-14 \\ 
$\texttt{n3c6-b5}$ & 5005/3003 & 3003/0.002 &2 &\emph{0.00037}& 2e-16 &2 &0.00082& 2e-16 &1 &0.00058& 4e-16 &1 &\textbf{0.00021}& 2e-16 \\ 
$\texttt{n3c6-b6}$ & 6435/5005 & 5005/0.001 &2 &\emph{0.00053}& 6e-14 &2 &0.0012& 6e-14 &1 &0.00099& 1e-13 &1 &\textbf{0.00039}& 1e-13 \\ 
$\texttt{n4c5-b4}$ & 2852/1350 & 1350/0.004 &9 &\textbf{0.00065}& 1e-07 &9 &0.00096& 1e-07 &8 &0.0013& 6e-07 &8 &\emph{0.00081}& 6e-07 \\ 
$\texttt{n4c5-b5}$ & 4340/2852 & 2852/0.002 &10 &\emph{0.0018}& 8e-08 &10 &0.0019& 5e-08 &9 &0.0024& 5e-07 &9 &\textbf{0.0013}& 5e-07 \\ 
$\texttt{n4c5-b6}$ & 4735/4340 & 4340/0.002 &11 &\emph{0.0023}& 2e-07 &11 &0.0032& 1e-07 &10 &0.0037& 9e-07 &10 &\textbf{0.002}& 1e-06 \\ 
$\texttt{n4c6-b3}$ & 5970/1330 & 1330/0.003 &3 &\emph{0.0003}& 1e-16 &3 &0.00045& 2e-17 &2 &0.00067& 5e-16 &2 &\textbf{0.00027}& 5e-16 \\ 
$\texttt{n4c6-b4}$ & 20058/5970 & 5970/0.0008 &8 &\emph{0.0032}& 8e-09 &8 &0.0047& 6e-08 &7 &0.0051& 3e-07 &7 &\textbf{0.0031}& 3e-07 \\ 
$\texttt{rel7}$ & 21924/1045 & 1043/0.002 &0 &\textbf{0.0002}& 0 &0 &0.00087& 0 &0 &0.00039& 0 &0 &\emph{0.0003}& 0 \\ 
$\texttt{relat7b}$ & 21924/1045 & 1043/0.004 &0 &\textbf{0.00025}& 0 &0 &0.0012& 0 &0 &0.00044& 0 &0 &\emph{0.00027}& 0 \\ 
$\texttt{relat7}$ & 21924/1045 & 1043/0.004 &0 &\emph{0.00028}& 0 &0 &0.00053& 0 &0 &0.0004& 0 &0 &\textbf{0.00028}& 0 \\ 
$\texttt{mesh\_deform}$ & 234023/9393 & 9393/0.0004 &1040 &3.5& 8e-07 &424 &\textbf{1.4}& 9e-07 &922 &3.1& 1e-06 &942 &\emph{2.9}& 1e-06 \\ 
$\texttt{162bit}$ & 3606/3597 & 3460/0.003 &2174 &0.48& 9e-07 &422 &\textbf{0.1}& 1e-06 &1597 &0.47& 1e-06 &1685 &\emph{0.34}& 1e-06 \\ 
$\texttt{176bit}$ & 7441/7431 & 7110/0.001 &3268 &1.3& 9e-07 &423 &\textbf{0.18}& 1e-06 &2369 &1.3& 1e-06 &2490 &\emph{0.98}& 1e-06 \\ 
$\texttt{specular}$ & 477976/1600 & 1442/0.01 &$\textnormal{2600}^{\dagger}$ &$\textnormal{\emph{42}}^{\dagger}$& $ \textnormal{0.0003}^{\dagger}$ &118 &\textbf{1.9}& 8e-07 &$\textnormal{2600}^{\dagger}$ &$\textnormal{43}^{\dagger}$& $ \textnormal{2e-05}^{\dagger}$ &$\textnormal{2600}^{\dagger}$ &$\textnormal{43}^{\dagger}$& $ \textnormal{2e-05}^{\dagger}$ \\ 
 \hline
 \end{tabular}
 \hss}
\end{table}

\subsection{Experiment III}
\label{sec:EXII}

\jbbR{In this experiment the matrices are large with $m>n$ and $10^4 \le n \le 10^7$.
The right-hand side $\bbold$ and starting vector $\x_0$ are initialized as in Experiment I. Because computing full random normal sketching matrices is not feasible for
these large matrices, we use $\texttt{sprandn}$ instead of $\texttt{randn}$ in a randomized implementation
of \eqref{eq:pkLong}. The parameter $r$ is set as follows: If $n>10^5$ then $r=5$ else $r=50$. 
Convergence is determined if $\|\A\xk-\bbold\|_2 \le \epsilon $ with $\epsilon=10^{-2}$. The iteration limit is 500 and outcomes are reported in Table \ref{tab:exp2}.}

\begin{table}[p]  
 \captionsetup{width=1\linewidth}
\caption{\jbbR{Experiment III compares 4 solvers on 51 linear systems from the SuiteSparse Matrix Collection \cite{SuiteSparseMatrix} with stopping tolerance $\epsilon = 10^{-2}$ and iteration limit $500$. 
\jbrv{In column 3, ``Rank'' is the structural rank of the matrix and} ``Dty" is the density of a particular matrix $ \b{A} $ calculated as $ \text{Dty} = \frac{\text{nnz}(\b{A})}{m \cdot n} $.
Entries with 
superscript $^{\dagger}$ 
denote problems for which the solver did not converge to the specified tolerance. 
Bold entries mark the
fastest times, while second fastest times are italicized.
}}
\label{tab:exp2}

\setlength{\tabcolsep}{2pt} 

\scriptsize
\hbox to 1.00\textwidth{\hss 
\begin{tabular}{|l | c | c | c c c | c c c | c c c | c c c |} 
		\hline
		\multirow{2}{*}{Problem} & \multirow{2}{*}{$m$/$n$} & \multirow{2}{*}{\text{\jbrv{Rank}}/\text{Dty}} & 
        \multicolumn{3}{c|}{\jbbo{Rand. Proj.}}  & 
		\multicolumn{3}{c|}{\jbbo{PLSS}} &  		
		\multicolumn{3}{c|}{\jbbo{PLSS W}} &  	
		\multicolumn{3}{c|}{\jbbo{LSQR}} 	\\
		\cline{4-15}
		& & 
		& It &  Sec & Res
		& It &  Sec & Res
		& It &  Sec & Res
		& It &  Sec & Res \\
		\hline
$\texttt{graphics}$ & 29493/11822 & 11822/0.0003 &$\textnormal{500}^{\dagger}$ &$\textnormal{10}^{\dagger}$& $ \textnormal{2}^{\dagger}$ &$\textnormal{500}^{\dagger}$ &$\textnormal{\textbf{0.18}}^{\dagger}$& $ \textnormal{0.05}^{\dagger}$ &$\textnormal{500}^{\dagger}$ &$\textnormal{\emph{0.2}}^{\dagger}$& $ \textnormal{0.01}^{\dagger}$ &$\textnormal{500}^{\dagger}$ &$\textnormal{0.22}^{\dagger}$& $ \textnormal{0.003}^{\dagger}$ \\ 
$\texttt{deltaX}$ & 68600/21961 & 21961/0.0002 &$\textnormal{500}^{\dagger}$ &$\textnormal{25}^{\dagger}$& $ \textnormal{0.08}^{\dagger}$ &$\textnormal{500}^{\dagger}$ &$\textnormal{\textbf{0.43}}^{\dagger}$& $ \textnormal{0.0002}^{\dagger}$ &$\textnormal{500}^{\dagger}$ &$\textnormal{\emph{0.46}}^{\dagger}$& $ \textnormal{5e-05}^{\dagger}$ &$\textnormal{500}^{\dagger}$ &$\textnormal{0.51}^{\dagger}$& $ \textnormal{2e-05}^{\dagger}$ \\ 
$\texttt{NotreDame\_actors}$ & 392400/127823 & 114762/3e-05 &$\textnormal{500}^{\dagger}$ &$\textnormal{2.5e+02}^{\dagger}$& $ \textnormal{0.1}^{\dagger}$ &$\textnormal{500}^{\dagger}$ &$\textnormal{\textbf{3.9}}^{\dagger}$& $ \textnormal{0.001}^{\dagger}$ &$\textnormal{500}^{\dagger}$ &$\textnormal{4.2}^{\dagger}$& $ \textnormal{1e-05}^{\dagger}$ &$\textnormal{500}^{\dagger}$ &$\textnormal{\emph{4.2}}^{\dagger}$& $ \textnormal{0.0001}^{\dagger}$ \\ 
$\texttt{ESOC}$ & 327062/37830 & 37349/0.0005 &$\textnormal{500}^{\dagger}$ &$\textnormal{2.8e+02}^{\dagger}$& $ \textnormal{0.09}^{\dagger}$ &$\textnormal{500}^{\dagger}$ &$\textnormal{\emph{8.9}}^{\dagger}$& $ \textnormal{0.05}^{\dagger}$ &$\textnormal{500}^{\dagger}$ &$\textnormal{9}^{\dagger}$& $ \textnormal{0.03}^{\dagger}$ &$\textnormal{500}^{\dagger}$ &$\textnormal{\textbf{8.7}}^{\dagger}$& $ \textnormal{0.003}^{\dagger}$ \\ 
$\texttt{psse0}$ & 26722/11028 & 11028/0.0003 &$\textnormal{500}^{\dagger}$ &$\textnormal{9.5}^{\dagger}$& $ \textnormal{1}^{\dagger}$ &$\textnormal{500}^{\dagger}$ &$\textnormal{\textbf{0.19}}^{\dagger}$& $ \textnormal{0.8}^{\dagger}$ &$\textnormal{500}^{\dagger}$ &$\textnormal{\emph{0.21}}^{\dagger}$& $ \textnormal{0.6}^{\dagger}$ &$\textnormal{500}^{\dagger}$ &$\textnormal{0.22}^{\dagger}$& $ \textnormal{0.06}^{\dagger}$ \\ 
$\texttt{psse1}$ & 14318/11028 & 11028/0.0004 &$\textnormal{500}^{\dagger}$ &$\textnormal{5.8}^{\dagger}$& $ \textnormal{1}^{\dagger}$ &$\textnormal{500}^{\dagger}$ &$\textnormal{\textbf{0.15}}^{\dagger}$& $ \textnormal{6}^{\dagger}$ &$\textnormal{500}^{\dagger}$ &$\textnormal{\emph{0.18}}^{\dagger}$& $ \textnormal{1}^{\dagger}$ &$\textnormal{500}^{\dagger}$ &$\textnormal{0.19}^{\dagger}$& $ \textnormal{0.07}^{\dagger}$ \\ 
$\texttt{psse2}$ & 28634/11028 & 11028/0.0004 &$\textnormal{500}^{\dagger}$ &$\textnormal{11}^{\dagger}$& $ \textnormal{1}^{\dagger}$ &$\textnormal{500}^{\dagger}$ &$\textnormal{\textbf{0.22}}^{\dagger}$& $ \textnormal{6}^{\dagger}$ &$\textnormal{500}^{\dagger}$ &$\textnormal{\emph{0.25}}^{\dagger}$& $ \textnormal{0.4}^{\dagger}$ &$\textnormal{500}^{\dagger}$ &$\textnormal{0.25}^{\dagger}$& $ \textnormal{0.08}^{\dagger}$ \\ 
$\texttt{Rucci1}$ & 1977885/109900 & 109900/4e-05 &$\textnormal{500}^{\dagger}$ &$\textnormal{1e+03}^{\dagger}$& $ \textnormal{1}^{\dagger}$ &$\textnormal{500}^{\dagger}$ &$\textnormal{\emph{17}}^{\dagger}$& $ \textnormal{3}^{\dagger}$ &$\textnormal{500}^{\dagger}$ &$\textnormal{17}^{\dagger}$& $ \textnormal{2}^{\dagger}$ &$\textnormal{388}^{\dagger}$ &$\textnormal{\textbf{13}}^{\dagger}$& $ \textnormal{0.2}^{\dagger}$ \\ 
$\texttt{Maragal\_6}$ & 21255/10152 & 10052/0.002 &$\textnormal{500}^{\dagger}$ &$\textnormal{12}^{\dagger}$& $ \textnormal{0.4}^{\dagger}$ &$\textnormal{500}^{\dagger}$ &$\textnormal{\textbf{0.46}}^{\dagger}$& $ \textnormal{0.01}^{\dagger}$ &$\textnormal{500}^{\dagger}$ &$\textnormal{0.51}^{\dagger}$& $ \textnormal{0.002}^{\dagger}$ &$\textnormal{500}^{\dagger}$ &$\textnormal{\emph{0.47}}^{\dagger}$& $ \textnormal{0.001}^{\dagger}$ \\ 
$\texttt{Maragal\_7}$ & 46845/26564 & 25866/0.001 &$\textnormal{500}^{\dagger}$ &$\textnormal{29}^{\dagger}$& $ \textnormal{0.6}^{\dagger}$ &$\textnormal{500}^{\dagger}$ &$\textnormal{\emph{1.3}}^{\dagger}$& $ \textnormal{0.01}^{\dagger}$ &$\textnormal{500}^{\dagger}$ &$\textnormal{\textbf{1.3}}^{\dagger}$& $ \textnormal{0.003}^{\dagger}$ &$\textnormal{500}^{\dagger}$ &$\textnormal{1.4}^{\dagger}$& $ \textnormal{0.001}^{\dagger}$ \\ 
$\texttt{Franz11}$ & 47104/30144 & 30144/0.0002 &$\textnormal{500}^{\dagger}$ &$\textnormal{29}^{\dagger}$& $ \textnormal{0.3}^{\dagger}$ &7 &\textbf{0.023}& 6e-06 &5 &0.029& 4e-05 &7 &\emph{0.024}& 6e-06 \\ 
$\texttt{IG5-16}$ & 18846/18485 & 9519/0.002 &$\textnormal{500}^{\dagger}$ &$\textnormal{23}^{\dagger}$& $ \textnormal{0.02}^{\dagger}$ &$\textnormal{500}^{\dagger}$ &$\textnormal{\textbf{1.3}}^{\dagger}$& $ \textnormal{0.001}^{\dagger}$ &$\textnormal{500}^{\dagger}$ &$\textnormal{\emph{1.4}}^{\dagger}$& $ \textnormal{1e-05}^{\dagger}$ &$\textnormal{500}^{\dagger}$ &$\textnormal{1.5}^{\dagger}$& $ \textnormal{0.0001}^{\dagger}$ \\ 
$\texttt{IG5-17}$ & 30162/27944 & 14060/0.001 &$\textnormal{500}^{\dagger}$ &$\textnormal{49}^{\dagger}$& $ \textnormal{0.03}^{\dagger}$ &$\textnormal{500}^{\dagger}$ &$\textnormal{2}^{\dagger}$& $ \textnormal{0.001}^{\dagger}$ &$\textnormal{500}^{\dagger}$ &$\textnormal{\emph{1.6}}^{\dagger}$& $ \textnormal{9e-06}^{\dagger}$ &$\textnormal{500}^{\dagger}$ &$\textnormal{\textbf{1.3}}^{\dagger}$& $ \textnormal{0.0001}^{\dagger}$ \\ 
$\texttt{IG5-18}$ & 47894/41550 & 20818/0.0009 &$\textnormal{500}^{\dagger}$ &$\textnormal{43}^{\dagger}$& $ \textnormal{0.04}^{\dagger}$ &$\textnormal{500}^{\dagger}$ &$\textnormal{\textbf{2}}^{\dagger}$& $ \textnormal{0.001}^{\dagger}$ &$\textnormal{500}^{\dagger}$ &$\textnormal{2.3}^{\dagger}$& $ \textnormal{1e-05}^{\dagger}$ &$\textnormal{500}^{\dagger}$ &$\textnormal{\emph{2.2}}^{\dagger}$& $ \textnormal{0.0001}^{\dagger}$ \\ 
$\texttt{GL7d14}$ & 171375/47271 & 47266/0.0002 &$\textnormal{500}^{\dagger}$ &$\textnormal{1.1e+02}^{\dagger}$& $ \textnormal{0.7}^{\dagger}$ &43 &\emph{0.33}& 7e-06 &33 &\textbf{0.29}& 7e-06 &41 &0.33& 6e-06 \\ 
$\texttt{GL7d15}$ & 460261/171375 & 171373/8e-05 &$\textnormal{500}^{\dagger}$ &$\textnormal{4.4e+02}^{\dagger}$& $ \textnormal{0.9}^{\dagger}$ &59 &\emph{1.6}& 3e-06 &47 &\textbf{1.4}& 3e-06 &57 &1.7& 4e-06 \\ 
$\texttt{GL7d16}$ & 955128/460261 & 460091/3e-05 &$\textnormal{500}^{\dagger}$ &$\textnormal{1.4e+03}^{\dagger}$& $ \textnormal{1}^{\dagger}$ &58 &\emph{7.8}& 2e-06 &43 &\textbf{6.2}& 2e-06 &57 &8.1& 2e-06 \\ 
$\texttt{GL7d17}$ & 1548650/955128 & 954861/2e-05 &$\textnormal{500}^{\dagger}$ &$\textnormal{2.6e+03}^{\dagger}$& $ \textnormal{1}^{\dagger}$ &56 &\emph{12}& 1e-06 &45 &\textbf{11}& 2e-06 &55 &13& 2e-06 \\ 
$\texttt{GL7d18}$ & 1955309/1548650 & 1548499/1e-05 &$\textnormal{500}^{\dagger}$ &$\textnormal{3.6e+03}^{\dagger}$& $ \textnormal{1}^{\dagger}$ &74 &\emph{23}& 2e-06 &63 &\textbf{21}& 1e-06 &71 &23& 2e-06 \\ 
$\texttt{ch7-6-b4}$ & 15120/12600 & 12600/0.0004 &$\textnormal{500}^{\dagger}$ &$\textnormal{5.8}^{\dagger}$& $ \textnormal{0.2}^{\dagger}$ &13 &\textbf{0.0034}& 3e-05 &13 &0.0073& 3e-05 &12 &\emph{0.0043}& 8e-05 \\ 
$\texttt{ch7-8-b3}$ & 58800/11760 & 11760/0.0003 &$\textnormal{500}^{\dagger}$ &$\textnormal{21}^{\dagger}$& $ \textnormal{0.3}^{\dagger}$ &5 &\textbf{0.0037}& 0.0002 &5 &\emph{0.0049}& 0.0002 &5 &0.0052& 0.0002 \\ 
$\texttt{ch7-8-b4}$ & 141120/58800 & 58800/9e-05 &$\textnormal{500}^{\dagger}$ &$\textnormal{59}^{\dagger}$& $ \textnormal{0.7}^{\dagger}$ &9 &\textbf{0.019}& 2e-05 &9 &0.031& 2e-05 &9 &\emph{0.022}& 1e-05 \\ 
$\texttt{ch7-9-b3}$ & 105840/17640 & 17640/0.0002 &$\textnormal{500}^{\dagger}$ &$\textnormal{38}^{\dagger}$& $ \textnormal{0.5}^{\dagger}$ &5 &\textbf{0.0065}& 0.0001 &5 &0.013& 0.0001 &5 &\emph{0.008}& 0.0001 \\ 
$\texttt{ch7-9-b4}$ & 317520/105840 & 105840/5e-05 &$\textnormal{500}^{\dagger}$ &$\textnormal{1.3e+02}^{\dagger}$& $ \textnormal{0.8}^{\dagger}$ &9 &\textbf{0.046}& 8e-06 &9 &0.068& 8e-06 &9 &\emph{0.054}& 8e-06 \\ 
$\texttt{ch7-9-b5}$ & 423360/317520 & 317520/2e-05 &$\textnormal{500}^{\dagger}$ &$\textnormal{2.3e+02}^{\dagger}$& $ \textnormal{0.9}^{\dagger}$ &12 &\textbf{0.1}& 0.0002 &12 &0.16& 0.0002 &11 &\emph{0.13}& 0.0004 \\ 
$\texttt{ch8-8-b3}$ & 117600/18816 & 18816/0.0002 &$\textnormal{500}^{\dagger}$ &$\textnormal{47}^{\dagger}$& $ \textnormal{0.5}^{\dagger}$ &5 &\textbf{0.008}& 5e-05 &5 &0.013& 5e-05 &5 &\emph{0.011}& 5e-05 \\ 
$\texttt{ch8-8-b4}$ & 376320/117600 & 117600/4e-05 &$\textnormal{500}^{\dagger}$ &$\textnormal{2.1e+02}^{\dagger}$& $ \textnormal{0.8}^{\dagger}$ &8 &\textbf{0.05}& 5e-06 &8 &0.074& 5e-06 &8 &\emph{0.059}& 5e-06 \\ 
$\texttt{ch8-8-b5}$ & 564480/376320 & 376320/2e-05 &$\textnormal{500}^{\dagger}$ &$\textnormal{3.7e+02}^{\dagger}$& $ \textnormal{1}^{\dagger}$ &9 &\textbf{0.12}& 0.0003 &9 &0.18& 0.0003 &9 &\emph{0.17}& 0.0003 \\ 
$\texttt{D6-6}$ & 120576/23740 & 18660/5e-05 &$\textnormal{500}^{\dagger}$ &$\textnormal{44}^{\dagger}$& $ \textnormal{0.5}^{\dagger}$ &18 &\textbf{0.017}& 4e-05 &18 &0.023& 3e-05 &18 &\emph{0.02}& 3e-05 \\ 
$\texttt{mk12-b3}$ & 51975/13860 & 13860/0.0003 &$\textnormal{500}^{\dagger}$ &$\textnormal{20}^{\dagger}$& $ \textnormal{0.4}^{\dagger}$ &6 &\textbf{0.0045}& 3e-05 &6 &\emph{0.0058}& 3e-05 &6 &0.0064& 3e-05 \\ 
$\texttt{mk12-b4}$ & 62370/51975 & 51975/0.0001 &$\textnormal{500}^{\dagger}$ &$\textnormal{30}^{\dagger}$& $ \textnormal{0.6}^{\dagger}$ &11 &\textbf{0.012}& 3e-16 &11 &0.017& 3e-16 &11 &\emph{0.015}& 1e-14 \\ 
$\texttt{n4c6-b5}$ & 51813/20058 & 20058/0.0003 &$\textnormal{500}^{\dagger}$ &$\textnormal{22}^{\dagger}$& $ \textnormal{0.5}^{\dagger}$ &2 &\textbf{0.0025}& 2e-16 &4 &0.0065& 9e-05 &2 &\emph{0.0028}& 1e-15 \\ 
$\texttt{n4c6-b6}$ & 104115/51813 & 51813/0.0001 &$\textnormal{500}^{\dagger}$ &$\textnormal{50}^{\dagger}$& $ \textnormal{0.7}^{\dagger}$ &6 &\textbf{0.014}& 2e-05 &6 &0.024& 2e-05 &6 &\emph{0.018}& 2e-05 \\ 
$\texttt{n4c6-b7}$ & 163215/104115 & 104115/8e-05 &$\textnormal{500}^{\dagger}$ &$\textnormal{78}^{\dagger}$& $ \textnormal{0.9}^{\dagger}$ &5 &\textbf{0.022}& 0.0001 &5 &0.039& 8e-05 &5 &\emph{0.027}& 0.0001 \\ 
$\texttt{n4c6-b8}$ & 198895/163215 & 163215/6e-05 &$\textnormal{500}^{\dagger}$ &$\textnormal{1e+02}^{\dagger}$& $ \textnormal{0.9}^{\dagger}$ &9 &\textbf{0.052}& 8e-06 &9 &0.084& 7e-06 &9 &\emph{0.067}& 8e-06 \\ 
$\texttt{shar\_te2-b2}$ & 200200/17160 & 17160/0.0002 &$\textnormal{500}^{\dagger}$ &$\textnormal{73}^{\dagger}$& $ \textnormal{0.4}^{\dagger}$ &7 &\textbf{0.015}& 1e-05 &7 &0.02& 1e-05 &7 &\emph{0.015}& 1e-05 \\ 
$\texttt{kneser\_10\_4\_1}$ & 349651/330751 & 323401/9e-06 &$\textnormal{500}^{\dagger}$ &$\textnormal{2e+02}^{\dagger}$& $ \textnormal{0.6}^{\dagger}$ &32 &\textbf{0.19}& 2e-06 &31 &\emph{0.25}& 2e-06 &31 &0.25& 2e-06 \\ 
$\texttt{kneser\_8\_3\_1}$ & 15737/15681 & 14897/0.0002 &$\textnormal{500}^{\dagger}$ &$\textnormal{6.3}^{\dagger}$& $ \textnormal{0.09}^{\dagger}$ &28 &\textbf{0.0063}& 1e-05 &26 &0.01& 1e-05 &26 &\emph{0.0082}& 2e-05 \\ 
$\texttt{wheel\_601}$ & 902103/723605 & 723005/3e-06 &$\textnormal{500}^{\dagger}$ &$\textnormal{5.9e+02}^{\dagger}$& $ \textnormal{0.1}^{\dagger}$ &42 &\textbf{0.59}& 6e-07 &42 &\emph{0.78}& 6e-07 &42 &0.87& 7e-07 \\ 
$\texttt{rel8}$ & 345688/12347 & 12345/0.0002 &0 &0.28& 0 &0 &\textbf{0.0031}& 0 &0 &0.008& 0 &0 &\emph{0.0075}& 0 \\ 
$\texttt{rel9}$ & 9888048/274669 & 274667/9e-06 &0 &9.6& 0 &0 &\textbf{0.21}& 0 &0 &0.49& 0 &0 &\emph{0.23}& 0 \\ 
$\texttt{relat8}$ & 345688/12347 & 12345/0.0003 &0 &0.32& 0 &0 &\textbf{0.0044}& 0 &0 &0.029& 0 &0 &\emph{0.0046}& 0 \\ 
$\texttt{relat9}$ & 12360060/549336 & 274667/6e-06 &0 &13& 0 &0 &\textbf{0.4}& 0 &0 &0.85& 0 &0 &\emph{0.41}& 0 \\ 
$\texttt{sls}$ & 1748122/62729 & 62729/6e-05 &$\textnormal{500}^{\dagger}$ &$\textnormal{9.6e+02}^{\dagger}$& $ \textnormal{0.08}^{\dagger}$ &392 &16& 3e-06 &130 &\textbf{5.2}& 2e-06 &334 &\emph{13}& 3e-06 \\ 
$\texttt{image\_interp}$ & 240000/120000 & 120000/2e-05 &$\textnormal{500}^{\dagger}$ &$\textnormal{1e+02}^{\dagger}$& $ \textnormal{1}^{\dagger}$ &$\textnormal{500}^{\dagger}$ &$\textnormal{\textbf{1.3}}^{\dagger}$& $ \textnormal{0.3}^{\dagger}$ &$\textnormal{500}^{\dagger}$ &$\textnormal{\emph{1.5}}^{\dagger}$& $ \textnormal{0.3}^{\dagger}$ &$\textnormal{500}^{\dagger}$ &$\textnormal{1.7}^{\dagger}$& $ \textnormal{0.02}^{\dagger}$ \\ 
$\texttt{192bit}$ & 13691/13682 & 13006/0.0008 &$\textnormal{500}^{\dagger}$ &$\textnormal{7.5}^{\dagger}$& $ \textnormal{0.06}^{\dagger}$ &$\textnormal{500}^{\dagger}$ &$\textnormal{\emph{0.24}}^{\dagger}$& $ \textnormal{0.007}^{\dagger}$ &329 &\textbf{0.18}& 5e-06 &$\textnormal{500}^{\dagger}$ &$\textnormal{0.31}^{\dagger}$& $ \textnormal{0.0004}^{\dagger}$ \\ 
$\texttt{208bit}$ & 24430/24421 & 22981/0.0005 &$\textnormal{500}^{\dagger}$ &$\textnormal{16}^{\dagger}$& $ \textnormal{0.08}^{\dagger}$ &$\textnormal{500}^{\dagger}$ &$\textnormal{\emph{0.49}}^{\dagger}$& $ \textnormal{0.007}^{\dagger}$ &304 &\textbf{0.33}& 3e-06 &$\textnormal{500}^{\dagger}$ &$\textnormal{0.59}^{\dagger}$& $ \textnormal{0.0007}^{\dagger}$ \\ 
$\texttt{tomographic1}$ & 73159/59498 & 42208/0.0001 &$\textnormal{500}^{\dagger}$ &$\textnormal{39}^{\dagger}$& $ \textnormal{0.1}^{\dagger}$ &$\textnormal{500}^{\dagger}$ &$\textnormal{\textbf{0.95}}^{\dagger}$& $ \textnormal{0.001}^{\dagger}$ &$\textnormal{500}^{\dagger}$ &$\textnormal{\emph{1.1}}^{\dagger}$& $ \textnormal{0.0008}^{\dagger}$ &$\textnormal{500}^{\dagger}$ &$\textnormal{1.2}^{\dagger}$& $ \textnormal{9e-05}^{\dagger}$ \\ 
$\texttt{LargeRegFile}$ & 2111154/801374 & 801374/3e-06 &$\textnormal{500}^{\dagger}$ &$\textnormal{1.1e+03}^{\dagger}$& $ \textnormal{0.6}^{\dagger}$ &$\textnormal{500}^{\dagger}$ &$\textnormal{\emph{18}}^{\dagger}$& $ \textnormal{0.4}^{\dagger}$ &53 &\textbf{2.1}& 4e-06 &$\textnormal{500}^{\dagger}$ &$\textnormal{20}^{\dagger}$& $ \textnormal{0.004}^{\dagger}$ \\ 
$\texttt{JP}$ & 87616/67320 & 26137/0.002 &$\textnormal{500}^{\dagger}$ &$\textnormal{2.3e+02}^{\dagger}$& $ \textnormal{0.03}^{\dagger}$ &$\textnormal{500}^{\dagger}$ &$\textnormal{\emph{17}}^{\dagger}$& $ \textnormal{0.009}^{\dagger}$ &$\textnormal{500}^{\dagger}$ &$\textnormal{17}^{\dagger}$& $ \textnormal{0.002}^{\dagger}$ &$\textnormal{500}^{\dagger}$ &$\textnormal{\textbf{17}}^{\dagger}$& $ \textnormal{0.0005}^{\dagger}$ \\ 
$\texttt{Hardesty2}$ & 929901/303645 & 303645/1e-05 &$\textnormal{500}^{\dagger}$ &$\textnormal{4.4e+02}^{\dagger}$& $ \textnormal{1}^{\dagger}$ &$\textnormal{500}^{\dagger}$ &$\textnormal{\textbf{6.4}}^{\dagger}$& $ \textnormal{0.03}^{\dagger}$ &$\textnormal{500}^{\dagger}$ &$\textnormal{\emph{7}}^{\dagger}$& $ \textnormal{0.02}^{\dagger}$ &$\textnormal{500}^{\dagger}$ &$\textnormal{7.3}^{\dagger}$& $ \textnormal{0.002}^{\dagger}$ \\
  \hline
 \end{tabular}
 \hss}
\end{table} 



\subsection{Experiment IV}
\label{sec:EXIV}

\jbbRt{This experiment is on underdetermined systems $m<n$. 
The right-hand side $\bbold$, starting vector $\x_0$ and random sketching matrix are computed as in Experiment I.
\jbrv{The condition number of each matrix is in Table \ref{tab:cond}.}
Convergence is determined if $\|\A\xk-\bbold\|_2 \le \epsilon $ with $\epsilon=10^{-4}$. The iteration limit is $n+1500$ and outcomes are reported in Table~\ref{tab:exp3}.}

\begin{table}[p]  
 \captionsetup{width=1\linewidth}
\caption{\jbbRt{Experiment III compares 4 solvers on 42 linear systems from the SuiteSparse Matrix Collection \cite{SuiteSparseMatrix} with stopping tolerance $\epsilon = 10^{-4}$ and iteration limit $n+1500$. 
\jbrv{In column 3, ``Rank'' is the structural rank of the matrix and} ``Dty" is the density of a particular matrix $ \b{A} $ calculated as $ \text{Dty} = \frac{\text{nnz}(\b{A})}{m \cdot n} $.
Entries with 
superscript $^{\dagger}$ 
denote problems for which the solver did not converge to the specified tolerance.
Bold entries mark the
fastest times, while second fastest times are italicized.
}}
\label{tab:exp3}

\setlength{\tabcolsep}{2pt} 

\scriptsize
\hbox to 1.00\textwidth{\hss 
\begin{tabular}{|l | c | c | c c c | c c c | c c c | c c c |} 
		\hline
		\multirow{2}{*}{Problem} & \multirow{2}{*}{$m$/$n$} & \multirow{2}{*}{\text{\jbrv{Rank}}/\text{Dty}} & 
        \multicolumn{3}{c|}{\jbbo{Rand. Proj.}}  & 
		\multicolumn{3}{c|}{\jbbo{PLSS}} &  		
		\multicolumn{3}{c|}{\jbbo{PLSS W}} &  	
		\multicolumn{3}{c|}{\jbbo{CRAIG}} 	\\
		\cline{4-15}
		& & 
		& It &  Sec & Res
		& It &  Sec & Res
		& It &  Sec & Res
		& It &  Sec & Res \\
		\hline
$\texttt{lp\_25fv47}$ & 821/1876 & 820/0.007 &$\textnormal{3376}^{\dagger}$ &$\textnormal{0.9}^{\dagger}$& $ \textnormal{0.04}^{\dagger}$ &$\textnormal{3376}^{\dagger}$ &$\textnormal{\emph{0.1}}^{\dagger}$& $ \textnormal{4e-05}^{\dagger}$ &1697 &\textbf{0.065}& 4e-08 &$\textnormal{3376}^{\dagger}$ &$\textnormal{0.11}^{\dagger}$& $ \textnormal{0.0001}^{\dagger}$ \\ 
$\texttt{lp\_bnl1}$ & 643/1586 & 642/0.005 &$\textnormal{3086}^{\dagger}$ &$\textnormal{0.58}^{\dagger}$& $ \textnormal{0.01}^{\dagger}$ &1916 &\emph{0.036}& 5e-08 &398 &\textbf{0.013}& 6e-08 &1741 &0.037& 6e-08 \\ 
$\texttt{lp\_bnl2}$ & 2324/4486 & 2324/0.001 &$\textnormal{5986}^{\dagger}$ &$\textnormal{3.4}^{\dagger}$& $ \textnormal{0.01}^{\dagger}$ &5878 &\emph{0.28}& 4e-08 &1129 &\textbf{0.069}& 4e-08 &5417 &0.3& 4e-08 \\ 
$\texttt{lp\_cre\_a}$ & 3516/7248 & 3428/0.0007 &$\textnormal{8748}^{\dagger}$ &$\textnormal{6.7}^{\dagger}$& $ \textnormal{0.04}^{\dagger}$ &$\textnormal{8748}^{\dagger}$ &$\textnormal{0.73}^{\dagger}$& $ \textnormal{2e-05}^{\dagger}$ &3109 &\textbf{0.33}& 2e-08 &$\textnormal{8748}^{\dagger}$ &$\textnormal{\emph{0.72}}^{\dagger}$& $ \textnormal{0.0001}^{\dagger}$ \\ 
$\texttt{lp\_cre\_c}$ & 3068/6411 & 2986/0.0008 &$\textnormal{7911}^{\dagger}$ &$\textnormal{5.6}^{\dagger}$& $ \textnormal{0.05}^{\dagger}$ &$\textnormal{7911}^{\dagger}$ &$\textnormal{\emph{0.5}}^{\dagger}$& $ \textnormal{7e-05}^{\dagger}$ &2897 &\textbf{0.23}& 3e-08 &$\textnormal{7911}^{\dagger}$ &$\textnormal{0.57}^{\dagger}$& $ \textnormal{0.0002}^{\dagger}$ \\ 
$\texttt{lp\_cycle}$ & 1903/3371 & 1875/0.003 &$\textnormal{4871}^{\dagger}$ &$\textnormal{2.6}^{\dagger}$& $ \textnormal{0.005}^{\dagger}$ &$\textnormal{4871}^{\dagger}$ &$\textnormal{\textbf{0.28}}^{\dagger}$& $ \textnormal{0.0003}^{\dagger}$ &$\textnormal{4871}^{\dagger}$ &$\textnormal{0.32}^{\dagger}$& $ \textnormal{9e-05}^{\dagger}$ &$\textnormal{4871}^{\dagger}$ &$\textnormal{\emph{0.29}}^{\dagger}$& $ \textnormal{0.001}^{\dagger}$ \\ 
$\texttt{lp\_czprob}$ & 929/3562 & 929/0.003 &$\textnormal{5062}^{\dagger}$ &$\textnormal{1.8}^{\dagger}$& $ \textnormal{9e-05}^{\dagger}$ &123 &0.0049& 5e-09 &84 &\textbf{0.0043}& 6e-09 &109 &\emph{0.0044}& 4e-09 \\ 
$\texttt{lp\_d2q06c}$ & 2171/5831 & 2171/0.003 &$\textnormal{7331}^{\dagger}$ &$\textnormal{6}^{\dagger}$& $ \textnormal{0.02}^{\dagger}$ &$\textnormal{7331}^{\dagger}$ &$\textnormal{\emph{0.94}}^{\dagger}$& $ \textnormal{0.0001}^{\dagger}$ &$\textnormal{7331}^{\dagger}$ &$\textnormal{1.1}^{\dagger}$& $ \textnormal{0.0001}^{\dagger}$ &$\textnormal{7331}^{\dagger}$ &$\textnormal{\textbf{0.94}}^{\dagger}$& $ \textnormal{0.0009}^{\dagger}$ \\ 
$\texttt{lp\_d6cube}$ & 415/6184 & 404/0.01 &$\textnormal{7684}^{\dagger}$ &$\textnormal{4.8}^{\dagger}$& $ \textnormal{9e-05}^{\dagger}$ &240 &\emph{0.041}& 4e-09 &324 &0.06& 5e-09 &191 &\textbf{0.032}& 5e-09 \\ 
$\texttt{lp\_degen3}$ & 1503/2604 & 1503/0.006 &$\textnormal{4104}^{\dagger}$ &$\textnormal{2.2}^{\dagger}$& $ \textnormal{0.01}^{\dagger}$ &1150 &0.091& 2e-07 &522 &\textbf{0.046}& 3e-07 &1009 &\emph{0.082}& 3e-07 \\ 
$\texttt{lp\_fffff800}$ & 524/1028 & 524/0.01 &$\textnormal{2528}^{\dagger}$ &$\textnormal{0.43}^{\dagger}$& $ \textnormal{0.0009}^{\dagger}$ &$\textnormal{2528}^{\dagger}$ &$\textnormal{\textbf{0.053}}^{\dagger}$& $ \textnormal{7e-06}^{\dagger}$ &$\textnormal{2528}^{\dagger}$ &$\textnormal{0.062}^{\dagger}$& $ \textnormal{3e-09}^{\dagger}$ &$\textnormal{2528}^{\dagger}$ &$\textnormal{\emph{0.057}}^{\dagger}$& $ \textnormal{4e-05}^{\dagger}$ \\ 
$\texttt{lp\_finnis}$ & 497/1064 & 497/0.005 &$\textnormal{2564}^{\dagger}$ &$\textnormal{0.43}^{\dagger}$& $ \textnormal{0.005}^{\dagger}$ &346 &0.0059& 1e-07 &90 &\textbf{0.0019}& 5e-08 &330 &\emph{0.0058}& 2e-07 \\ 
$\texttt{lp\_fit1d}$ & 24/1049 & 24/0.5 &1148 &0.15& 1e-09 &89 &0.0029& 1e-10 &21 &\textbf{0.00098}& 4e-10 &65 &\emph{0.0021}& 2e-10 \\ 
$\texttt{lp\_fit1p}$ & 627/1677 & 627/0.009 &$\textnormal{3177}^{\dagger}$ &$\textnormal{0.74}^{\dagger}$& $ \textnormal{5e-06}^{\dagger}$ &120 &0.0041& 6e-09 &15 &\textbf{0.0008}& 7e-09 &103 &\emph{0.0039}& 1e-09 \\ 
$\texttt{lp\_ganges}$ & 1309/1706 & 1309/0.003 &$\textnormal{3206}^{\dagger}$ &$\textnormal{1.1}^{\dagger}$& $ \textnormal{0.03}^{\dagger}$ &114 &\emph{0.0033}& 9e-07 &93 &\textbf{0.0033}& 8e-07 &113 &0.0037& 1e-06 \\ 
$\texttt{lp\_gfrd\_pnc}$ & 616/1160 & 616/0.003 &$\textnormal{2660}^{\dagger}$ &$\textnormal{0.5}^{\dagger}$& $ \textnormal{0.0004}^{\dagger}$ &211 &\emph{0.0036}& 2e-09 &92 &\textbf{0.0021}& 2e-09 &203 &0.0039& 2e-09 \\ 
$\texttt{lp\_greenbea}$ & 2392/5598 & 2389/0.002 &$\textnormal{7098}^{\dagger}$ &$\textnormal{5.7}^{\dagger}$& $ \textnormal{0.007}^{\dagger}$ &3433 &\emph{0.3}& 8e-08 &1604 &\textbf{0.18}& 8e-08 &3151 &0.3& 8e-08 \\ 
$\texttt{lp\_greenbeb}$ & 2392/5598 & 2389/0.002 &$\textnormal{7098}^{\dagger}$ &$\textnormal{5.6}^{\dagger}$& $ \textnormal{0.007}^{\dagger}$ &3433 &\emph{0.3}& 8e-08 &1604 &\textbf{0.17}& 8e-08 &3151 &0.3& 8e-08 \\ 
$\texttt{lp\_ken\_07}$ & 2426/3602 & 2426/0.001 &$\textnormal{5102}^{\dagger}$ &$\textnormal{3.2}^{\dagger}$& $ \textnormal{0.01}^{\dagger}$ &171 &\textbf{0.0079}& 5e-08 &162 &0.0096& 2e-07 &167 &\emph{0.0087}& 2e-07 \\ 
$\texttt{lp\_maros}$ & 846/1966 & 846/0.006 &$\textnormal{3466}^{\dagger}$ &$\textnormal{1}^{\dagger}$& $ \textnormal{0.004}^{\dagger}$ &$\textnormal{3466}^{\dagger}$ &$\textnormal{\textbf{0.11}}^{\dagger}$& $ \textnormal{0.0001}^{\dagger}$ &$\textnormal{3466}^{\dagger}$ &$\textnormal{0.13}^{\dagger}$& $ \textnormal{3e-07}^{\dagger}$ &$\textnormal{3466}^{\dagger}$ &$\textnormal{\emph{0.13}}^{\dagger}$& $ \textnormal{0.002}^{\dagger}$ \\ 
$\texttt{lp\_maros\_r7}$ & 3136/9408 & 3136/0.005 &9409 &13& 6e-07 &14 &\emph{0.0039}& 3e-07 &14 &0.0058& 9e-08 &13 &\textbf{0.0032}& 6e-07 \\ 
$\texttt{lp\_modszk1}$ & 687/1620 & 686/0.003 &$\textnormal{3120}^{\dagger}$ &$\textnormal{0.72}^{\dagger}$& $ \textnormal{0.001}^{\dagger}$ &70 &\textbf{0.0016}& 9e-07 &61 &0.0017& 7e-07 &69 &\emph{0.0017}& 9e-07 \\ 
$\texttt{lp\_pds\_02}$ & 2953/7716 & 2953/0.0007 &$\textnormal{9216}^{\dagger}$ &$\textnormal{7.6}^{\dagger}$& $ \textnormal{0.004}^{\dagger}$ &119 &\textbf{0.0087}& 2e-07 &106 &0.011& 1e-07 &117 &\emph{0.01}& 2e-07 \\ 
$\texttt{lp\_perold}$ & 625/1506 & 625/0.007 &$\textnormal{3006}^{\dagger}$ &$\textnormal{0.68}^{\dagger}$& $ \textnormal{0.02}^{\dagger}$ &$\textnormal{3006}^{\dagger}$ &$\textnormal{\emph{0.07}}^{\dagger}$& $ \textnormal{0.0003}^{\dagger}$ &844 &\textbf{0.024}& 7e-10 &$\textnormal{3006}^{\dagger}$ &$\textnormal{0.079}^{\dagger}$& $ \textnormal{0.005}^{\dagger}$ \\ 
$\texttt{lp\_pilot}$ & 1441/4860 & 1441/0.006 &$\textnormal{6360}^{\dagger}$ &$\textnormal{4.4}^{\dagger}$& $ \textnormal{0.02}^{\dagger}$ &3344 &0.38& 6e-08 &660 &\textbf{0.087}& 6e-08 &3047 &\emph{0.35}& 5e-08 \\ 
$\texttt{lp\_pilot4}$ & 410/1123 & 410/0.01 &$\textnormal{2623}^{\dagger}$ &$\textnormal{0.44}^{\dagger}$& $ \textnormal{0.008}^{\dagger}$ &$\textnormal{2623}^{\dagger}$ &$\textnormal{\emph{0.05}}^{\dagger}$& $ \textnormal{2e-05}^{\dagger}$ &410 &\textbf{0.0095}& 5e-10 &$\textnormal{2623}^{\dagger}$ &$\textnormal{0.053}^{\dagger}$& $ \textnormal{0.0006}^{\dagger}$ \\ 
$\texttt{lp\_pilot87}$ & 2030/6680 & 2030/0.006 &$\textnormal{8180}^{\dagger}$ &$\textnormal{7.6}^{\dagger}$& $ \textnormal{0.05}^{\dagger}$ &7563 &1.2& 8e-09 &685 &\textbf{0.13}& 1e-08 &6880 &\emph{1}& 1e-08 \\ 
$\texttt{lp\_pilot\_ja}$ & 940/2267 & 940/0.007 &$\textnormal{3767}^{\dagger}$ &$\textnormal{1.3}^{\dagger}$& $ \textnormal{0.04}^{\dagger}$ &$\textnormal{3767}^{\dagger}$ &$\textnormal{\textbf{0.16}}^{\dagger}$& $ \textnormal{0.0003}^{\dagger}$ &$\textnormal{3767}^{\dagger}$ &$\textnormal{0.18}^{\dagger}$& $ \textnormal{0.0001}^{\dagger}$ &$\textnormal{3767}^{\dagger}$ &$\textnormal{\emph{0.16}}^{\dagger}$& $ \textnormal{0.003}^{\dagger}$ \\ 
$\texttt{lp\_pilot\_we}$ & 722/2928 & 722/0.004 &$\textnormal{4428}^{\dagger}$ &$\textnormal{1.2}^{\dagger}$& $ \textnormal{0.01}^{\dagger}$ &$\textnormal{4428}^{\dagger}$ &$\textnormal{\emph{0.16}}^{\dagger}$& $ \textnormal{0.0002}^{\dagger}$ &2711 &\textbf{0.13}& 6e-10 &$\textnormal{4428}^{\dagger}$ &$\textnormal{0.17}^{\dagger}$& $ \textnormal{0.003}^{\dagger}$ \\ 
$\texttt{lp\_pilotnov}$ & 975/2446 & 975/0.006 &$\textnormal{3946}^{\dagger}$ &$\textnormal{1.3}^{\dagger}$& $ \textnormal{0.02}^{\dagger}$ &$\textnormal{3946}^{\dagger}$ &$\textnormal{\textbf{0.15}}^{\dagger}$& $ \textnormal{0.0002}^{\dagger}$ &$\textnormal{3946}^{\dagger}$ &$\textnormal{0.18}^{\dagger}$& $ \textnormal{0.0001}^{\dagger}$ &$\textnormal{3946}^{\dagger}$ &$\textnormal{\emph{0.17}}^{\dagger}$& $ \textnormal{0.01}^{\dagger}$ \\ 
$\texttt{lp\_qap12}$ & 3192/8856 & 3192/0.001 &9114 &9& 2e-07 &7 &\emph{0.0011}& 9e-15 &7 &0.0016& 2e-15 &6 &\textbf{0.00091}& 8e-13 \\ 
$\texttt{lp\_qap8}$ & 912/1632 & 912/0.005 &2351 &0.65& 5e-07 &7 &\textbf{0.00021}& 7e-16 &7 &0.00033& 2e-15 &6 &\emph{0.00022}& 5e-14 \\ 
$\texttt{lp\_scfxm2}$ & 660/1200 & 660/0.007 &$\textnormal{2700}^{\dagger}$ &$\textnormal{0.56}^{\dagger}$& $ \textnormal{0.007}^{\dagger}$ &$\textnormal{2700}^{\dagger}$ &$\textnormal{\emph{0.06}}^{\dagger}$& $ \textnormal{2e-06}^{\dagger}$ &1091 &\textbf{0.03}& 1e-08 &$\textnormal{2700}^{\dagger}$ &$\textnormal{0.068}^{\dagger}$& $ \textnormal{8e-07}^{\dagger}$ \\ 
$\texttt{lp\_scfxm3}$ & 990/1800 & 990/0.005 &$\textnormal{3300}^{\dagger}$ &$\textnormal{0.97}^{\dagger}$& $ \textnormal{0.01}^{\dagger}$ &$\textnormal{3300}^{\dagger}$ &$\textnormal{\emph{0.097}}^{\dagger}$& $ \textnormal{2e-07}^{\dagger}$ &1113 &\textbf{0.039}& 1e-08 &$\textnormal{3300}^{\dagger}$ &$\textnormal{0.11}^{\dagger}$& $ \textnormal{1e-07}^{\dagger}$ \\ 
$\texttt{lp\_scrs8}$ & 490/1275 & 490/0.005 &$\textnormal{2775}^{\dagger}$ &$\textnormal{0.48}^{\dagger}$& $ \textnormal{0.004}^{\dagger}$ &$\textnormal{2775}^{\dagger}$ &$\textnormal{\emph{0.048}}^{\dagger}$& $ \textnormal{5e-06}^{\dagger}$ &747 &\textbf{0.017}& 5e-09 &$\textnormal{2775}^{\dagger}$ &$\textnormal{0.055}^{\dagger}$& $ \textnormal{2e-05}^{\dagger}$ \\ 
$\texttt{lp\_scsd6}$ & 147/1350 & 147/0.02 &$\textnormal{2850}^{\dagger}$ &$\textnormal{0.34}^{\dagger}$& $ \textnormal{0.0005}^{\dagger}$ &55 &\emph{0.0011}& 2e-06 &54 &0.0013& 1e-06 &54 &\textbf{0.0011}& 5e-06 \\ 
$\texttt{lp\_scsd8}$ & 397/2750 & 397/0.008 &$\textnormal{4250}^{\dagger}$ &$\textnormal{0.92}^{\dagger}$& $ \textnormal{0.0009}^{\dagger}$ &124 &\emph{0.0047}& 4e-06 &130 &0.0062& 5e-06 &123 &\textbf{0.0044}& 6e-06 \\ 
$\texttt{lp\_sctap2}$ & 1090/2500 & 1090/0.003 &$\textnormal{4000}^{\dagger}$ &$\textnormal{1.3}^{\dagger}$& $ \textnormal{0.02}^{\dagger}$ &785 &\emph{0.025}& 6e-08 &37 &\textbf{0.0016}& 4e-08 &750 &0.028& 7e-08 \\ 
$\texttt{lp\_sctap3}$ & 1480/3340 & 1480/0.002 &$\textnormal{4840}^{\dagger}$ &$\textnormal{2.3}^{\dagger}$& $ \textnormal{0.03}^{\dagger}$ &834 &\emph{0.036}& 6e-08 &40 &\textbf{0.0024}& 3e-08 &814 &0.039& 6e-08 \\ 
$\texttt{lp\_shell}$ & 536/1777 & 536/0.004 &$\textnormal{3277}^{\dagger}$ &$\textnormal{0.65}^{\dagger}$& $ \textnormal{0.004}^{\dagger}$ &85 &\textbf{0.0019}& 3e-07 &86 &0.0023& 2e-07 &82 &\emph{0.002}& 3e-07 \\ 
$\texttt{lp\_ship04l}$ & 402/2166 & 360/0.007 &$\textnormal{3666}^{\dagger}$ &$\textnormal{0.71}^{\dagger}$& $ \textnormal{0.002}^{\dagger}$ &70 &\textbf{0.0019}& 2e-07 &63 &0.0022& 8e-08 &69 &\emph{0.002}& 2e-07 \\ 
$\texttt{lp\_ship04s}$ & 402/1506 & 360/0.007 &$\textnormal{3006}^{\dagger}$ &$\textnormal{0.49}^{\dagger}$& $ \textnormal{0.003}^{\dagger}$ &90 &\emph{0.002}& 3e-07 &80 &0.0022& 1e-07 &88 &\textbf{0.002}& 3e-07 \\ 
   \hline
 \end{tabular}
 \hss}
\end{table} 

\clearpage

\subsection{Experiment V}
\label{sec:EXIII}

\jbbR{In this experiment we compare the proposed solvers with the implementations from \cite{GowerRichtarik15}.
The same test problems are used (i.e., \texttt{aloi-scale}, \texttt{covtype-libsvm}, \texttt{protein},
\texttt{SUSY}, and four additional ones). The quantities $\A$ and $\bbold$ are obtained from {\small LIBSVM} \cite{LIBSVM}. Convergence
is determined if the norm of residuals is less than or equal to $\epsilon=10^{-4}$. The outcomes are displayed in
Figure \ref{fig:EXIII}, with residuals for our proposed solvers and four methods from \cite[Section 7.3]{GowerRichtarik15}. }


\begin{figure}[t!]   
     \centering
     \begin{subfigure}[b]{0.49\textwidth}
         \centering
         \includegraphics[trim=0 200 50 200,width=\textwidth]{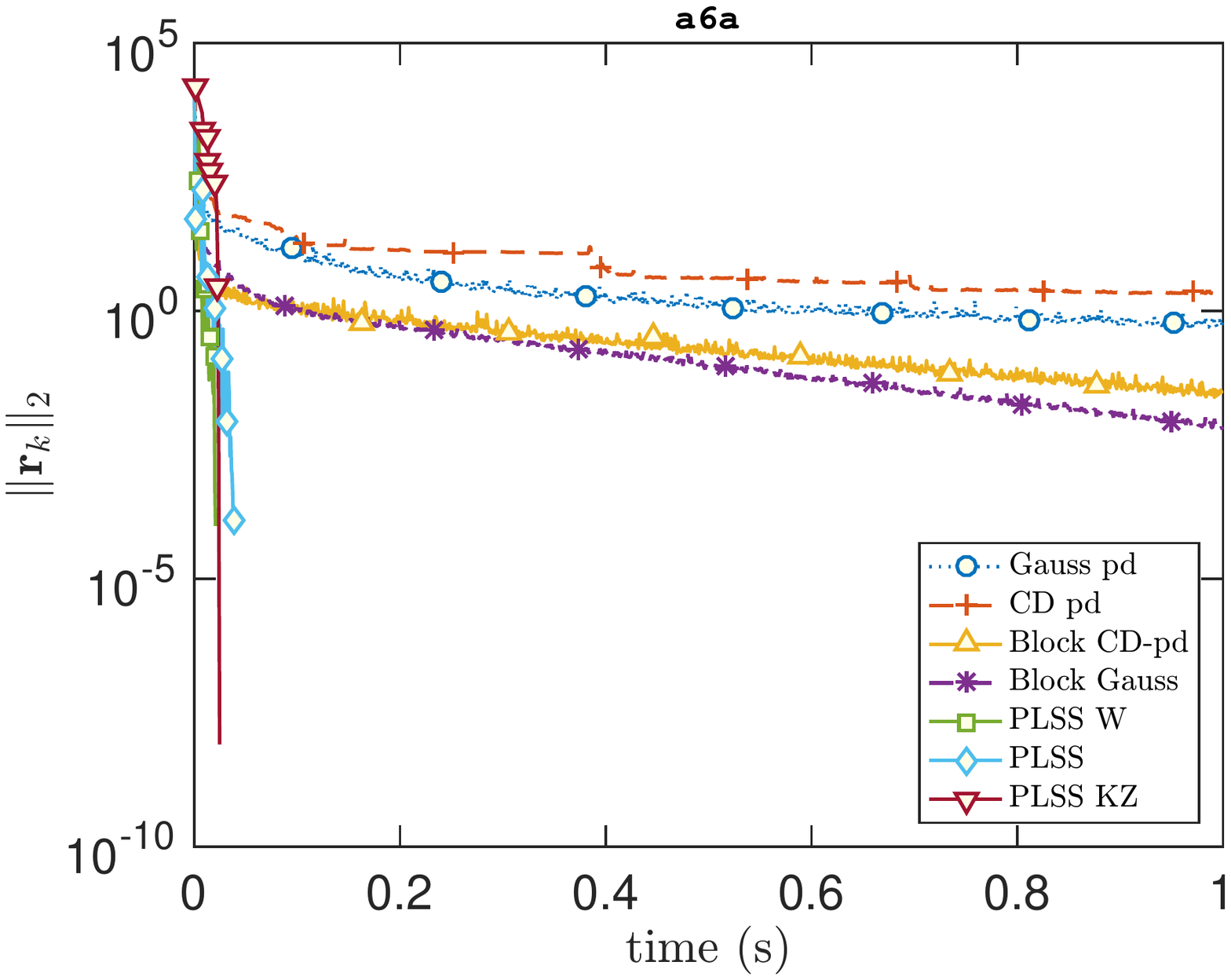}
         \label{fig:prob1}
     \end{subfigure}
     \hfill
     \begin{subfigure}[b]{0.49\textwidth}
         \centering
         \includegraphics[trim=50 200 0 200,width=\textwidth]{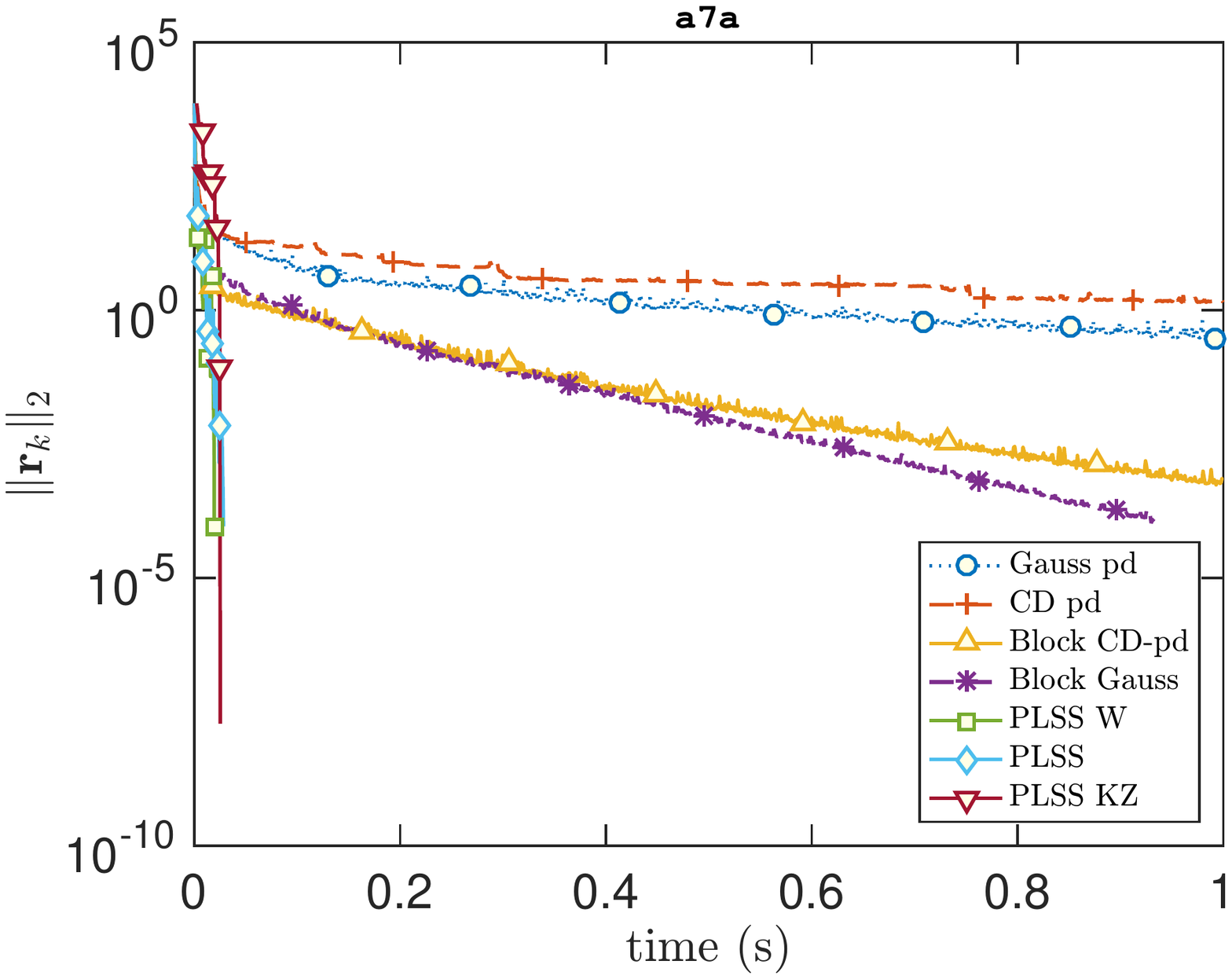}
         \label{fig:prob2}
     \end{subfigure}
     \hfill
     \begin{subfigure}[b]{0.49\textwidth}
         \centering
         \includegraphics[trim=0 200 50 200,width=\textwidth]{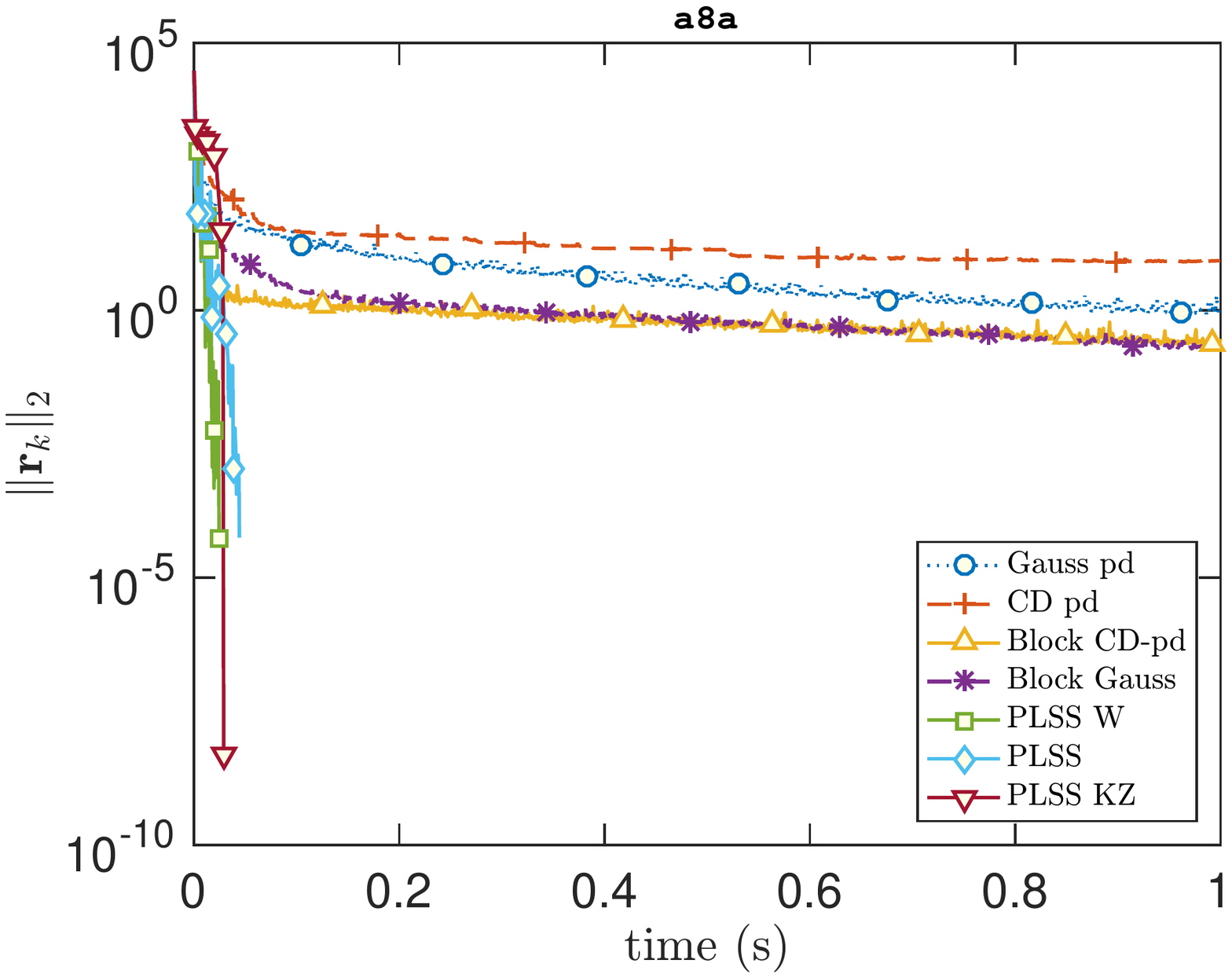}
         \label{fig:prob3}
     \end{subfigure}
     \hfill
     \begin{subfigure}[b]{0.49\textwidth}
         \centering
         \includegraphics[trim=50 200 0 200,width=\textwidth]{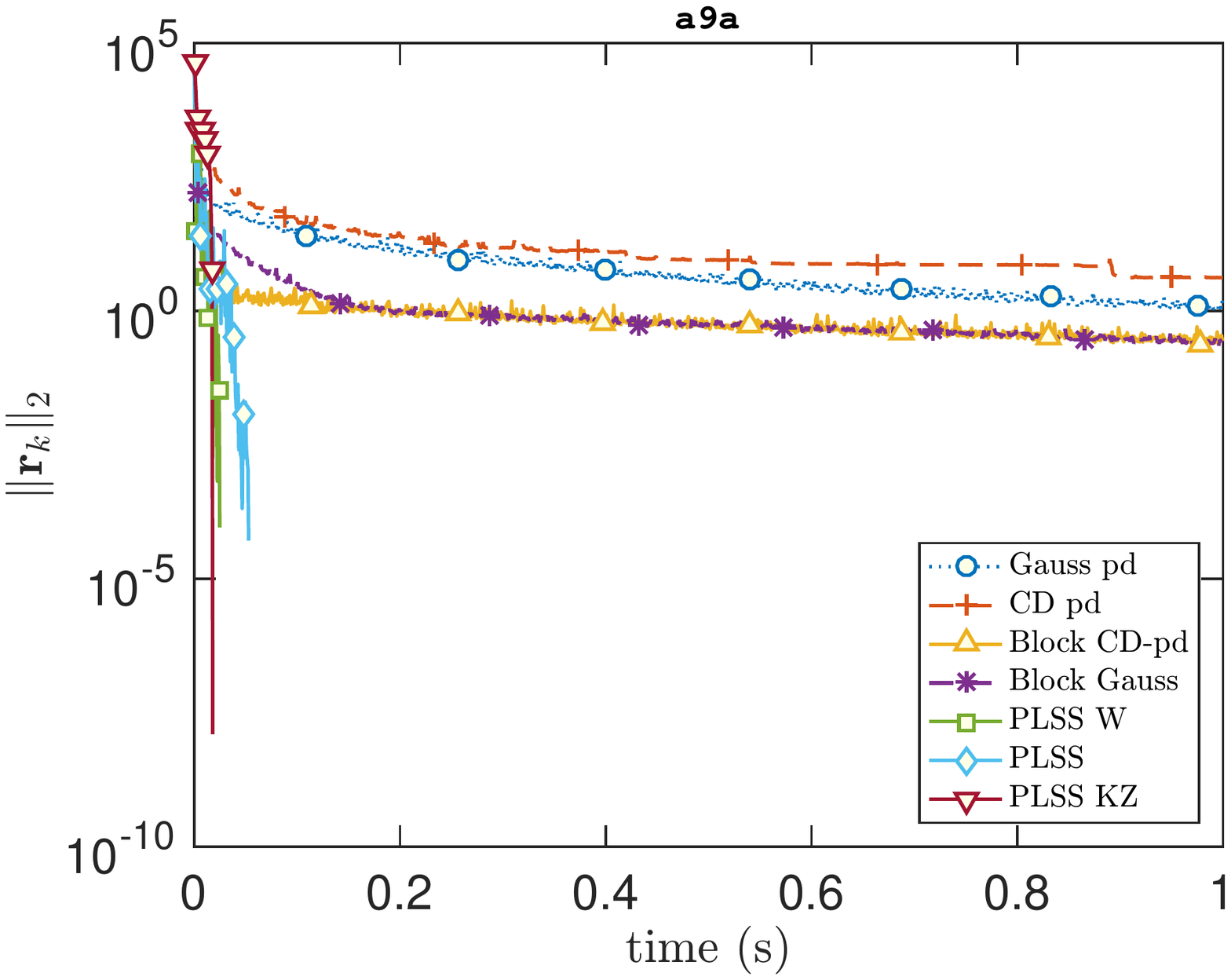}
         \label{fig:prob4}
     \end{subfigure}
     \hfill
     \begin{subfigure}[b]{0.49\textwidth}
         \centering
         \includegraphics[trim=0 200 50 200,width=\textwidth]{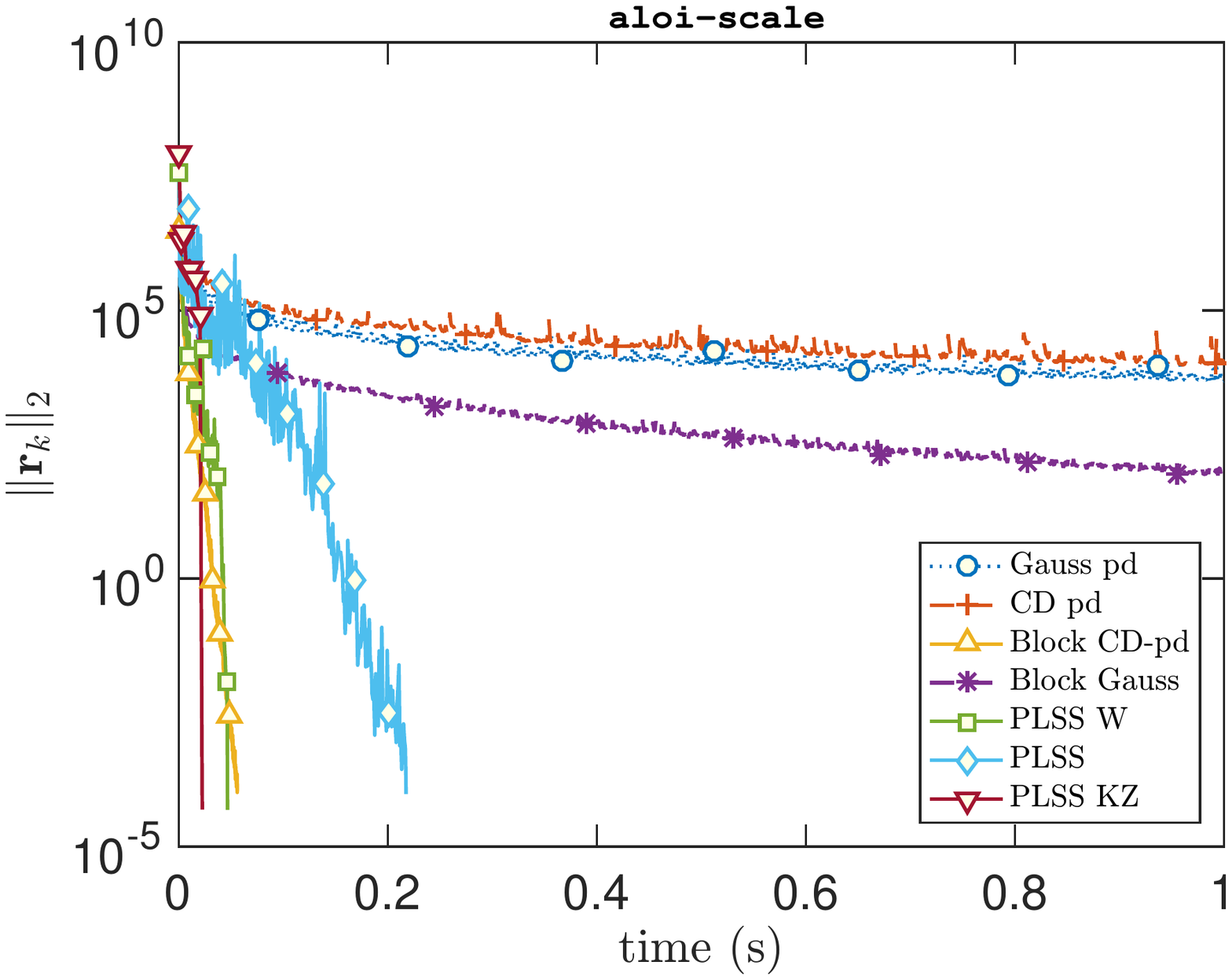}
         \label{fig:prob5}
     \end{subfigure}
     \hfill
     \begin{subfigure}[b]{0.49\textwidth}
         \centering
         \includegraphics[trim=50 200 0 200,width=\textwidth]{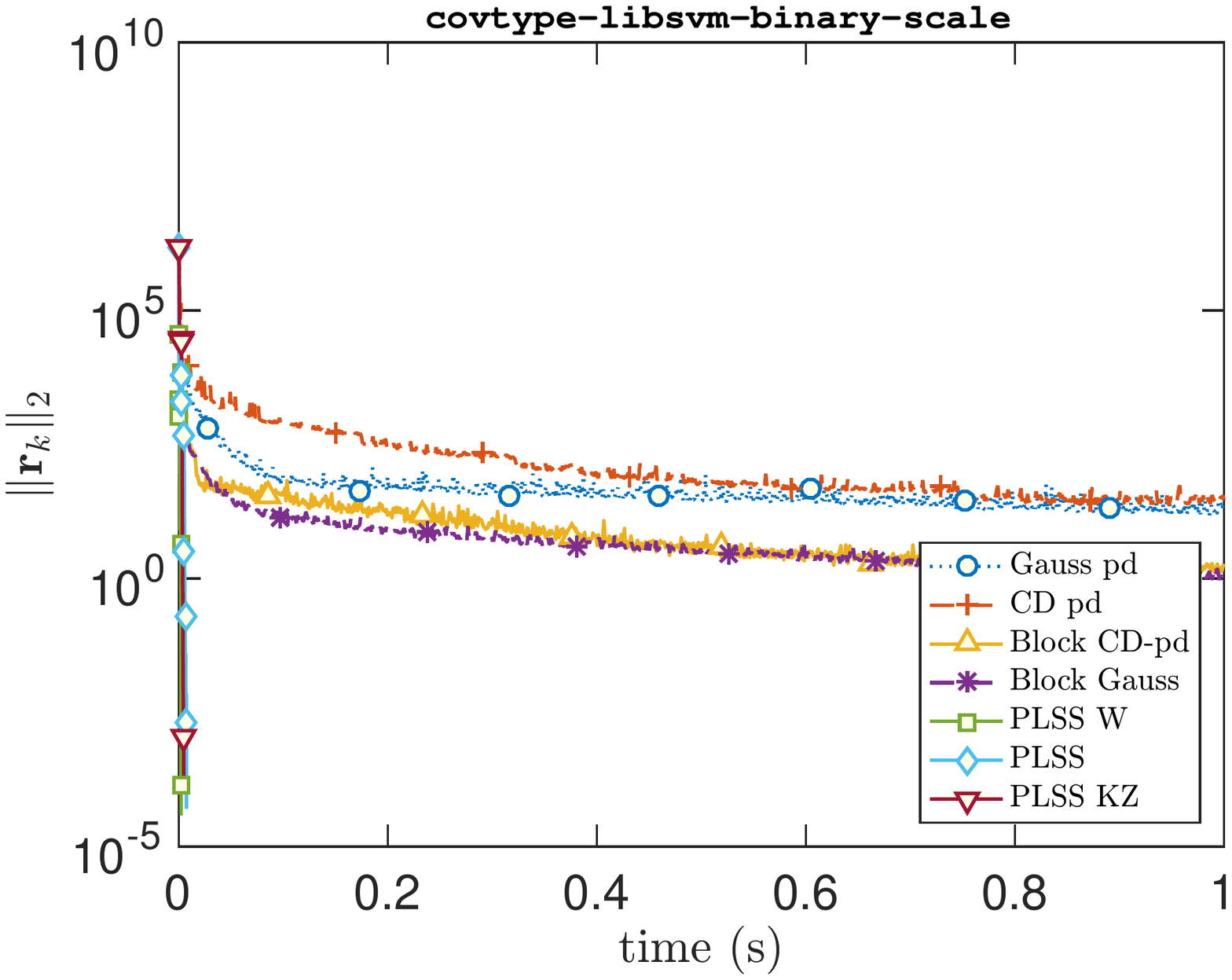}
         \label{fig:prob6}
     \end{subfigure}
     \hfill
     \begin{subfigure}[b]{0.49\textwidth}
         \centering
         \includegraphics[trim=0 200 50 200,width=\textwidth]{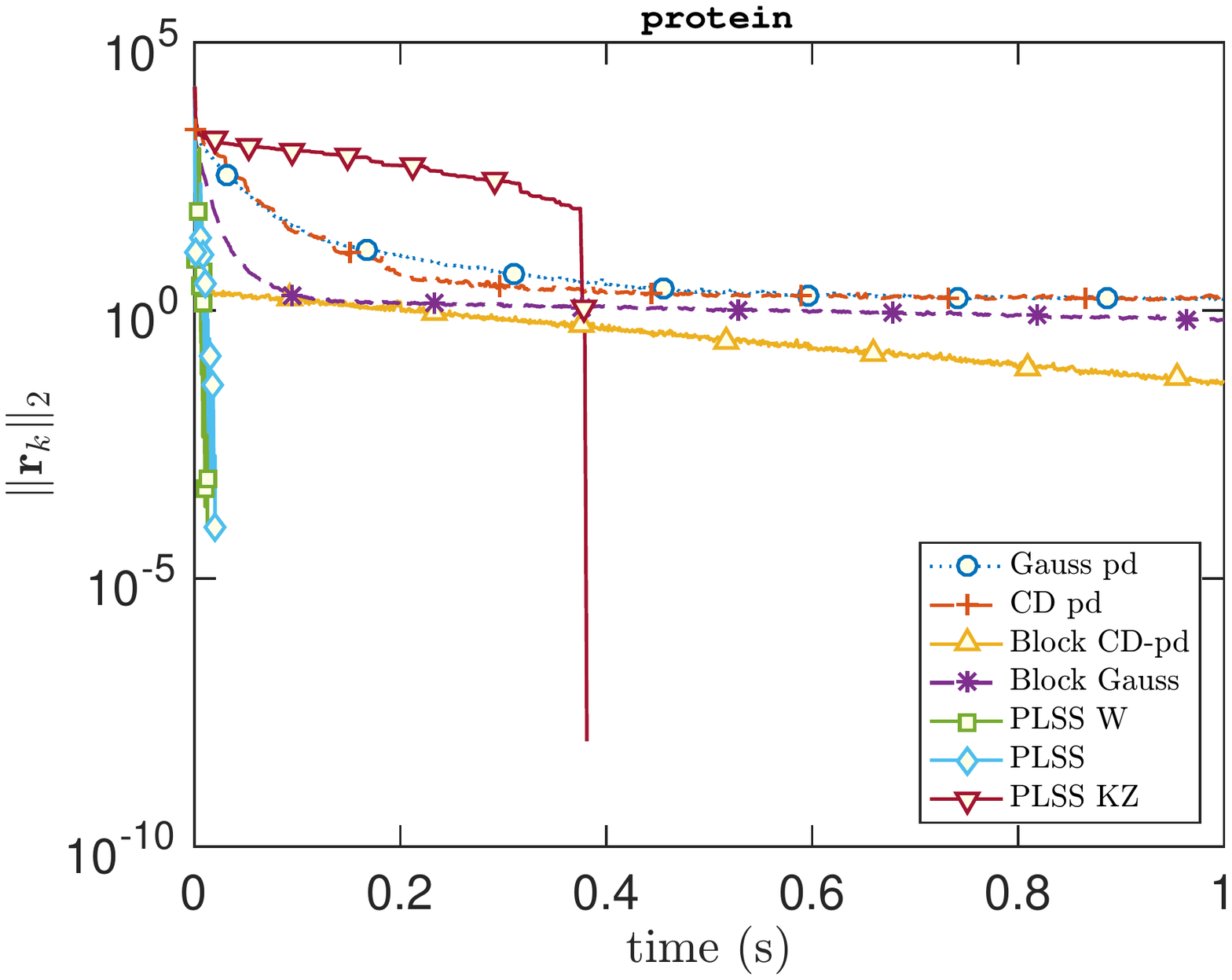}
         \label{fig:prob7}
     \end{subfigure}
     \hfill
     \begin{subfigure}[b]{0.49\textwidth}
         \centering
         \includegraphics[trim=50 200 0 200,width=\textwidth]{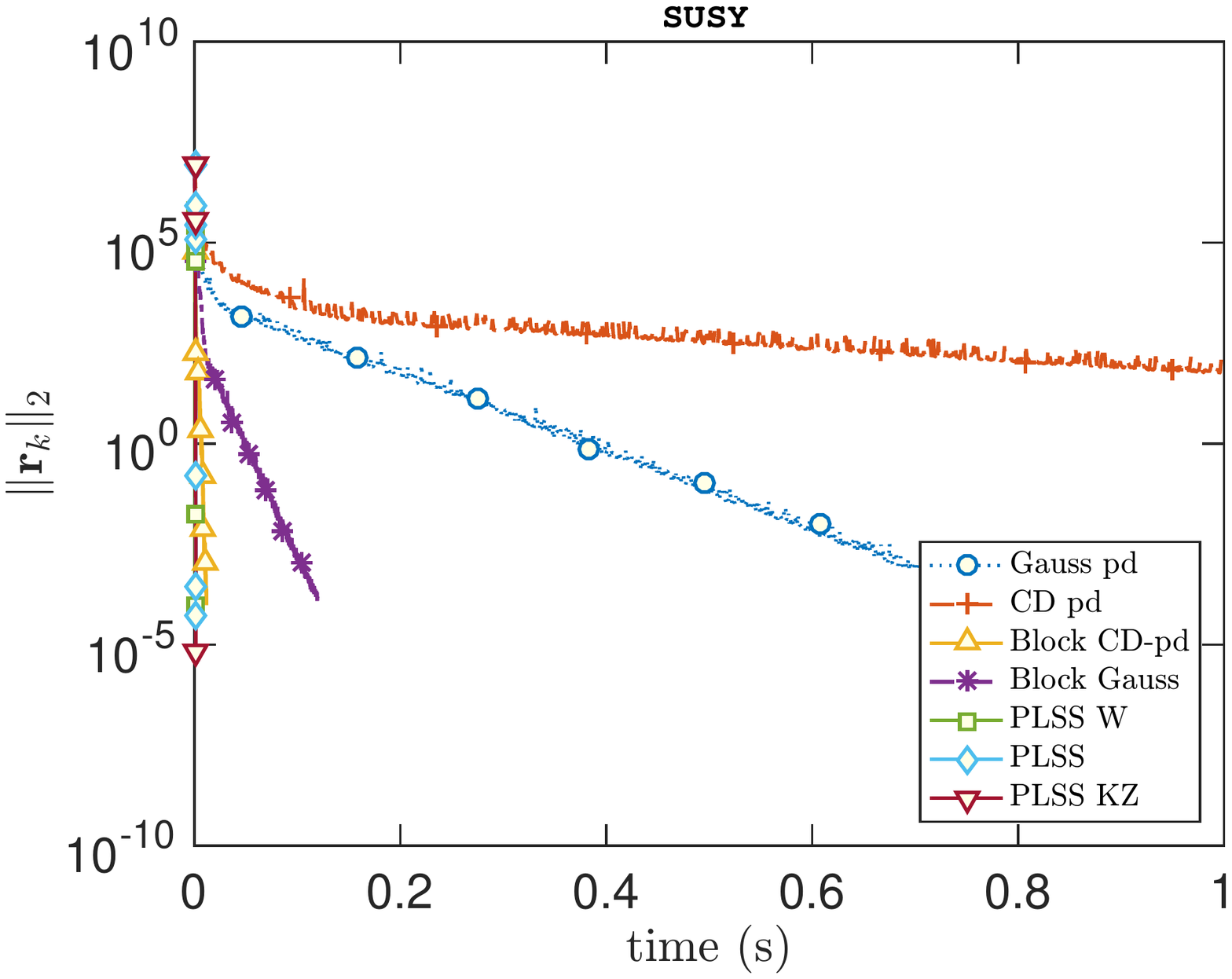}
         \label{fig:prob8}
     \end{subfigure}
        \caption{Comparison with randomized projection methods over a time interval of 1 second. The \texttt{error} represents the 2-norm of residuals. Four implementations of \cite{GowerRichtarik15}
        are included for reference.}
        \label{fig:EXIII}
\end{figure}

    \subsection{\jbbRo{Randomized projection}}
        \label{subsec:randproj}
    \jbbRo{For further comparison of our methods with randomized
    projections, we use the unsymmetric ill-conditioned `\texttt{sampling}' matrix from 
    MATLAB's matrix gallery with \texttt{n = 100;}
    \texttt{A = gallery(`sampling',n)}. The solution is the vector of all ones,
    and $\x_0 = \b{0}$. For this problem, $ \texttt{cond(A) = 2.7859e+17} $. In Figure \ref{fig:compALL} we compare Algorithm \ref{alg:PLSS} (with $\b{W} = \b{I}$) to updates \eqref{eq:pkLong} with random normal sketching
    matrices and varying dimensions of the sketch (i.e., $r$ varies). As expected, the convergence behavior of the randomized methods improves when $ r $ increases from 
    10 to 40. At the same time, {\small PLSS} uses a simple recursive update
    (with very low computational cost) and converges significantly more rapidly than any of
    the random sketches.}

    \begin{figure*}[t!]	
		\includegraphics[width=\textwidth]{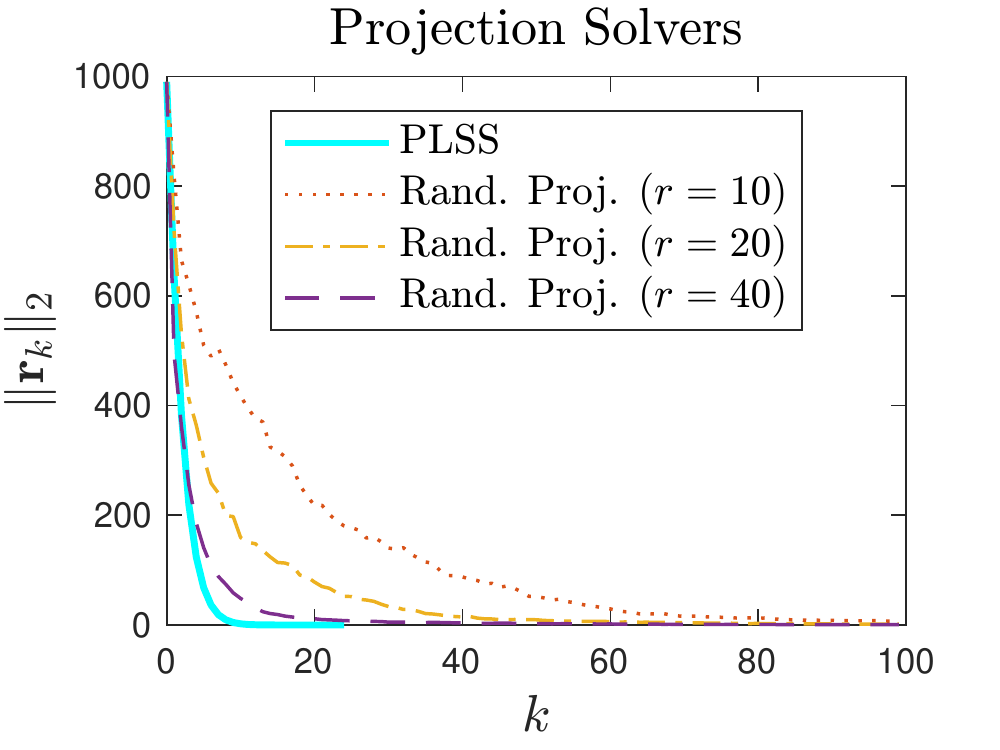}
		\caption{Comparison of {\small PLSS} with random normal projected solvers when the size of the sketch increases: $r=10,20,40$.}		
		\label{fig:compALL}       
\end{figure*}
	\section{Conclusions}
	\label{sec:conclusions}
	\jbbR{We develop an iterative projection method for solving consistent rectangular or square systems. Our method is based on appending one column each iteration to the sketching matrix. For full-rank sketches (common in practice) we prove that the underlying process terminates, with exact arithmetic, in a
	finite number of iterations. When the sketching matrix
	stores the history of all previous residuals, we develop a method with
	orthogonal residuals and updates. We include a parameter matrix that can be used to improve computations. Importantly, we derive a short recursive formula that is simple to implement, and an algorithm that updates only four vectors. In numerical experiments, including large sparse systems, our methods compare favorably to widely known methods ({\small LSQR}, {\small LSMR} or {\small CRAIG}) and to some existing randomized methods.}
	
	\appendix
	
	\section{Optimality}
	\label{app:A}
	\jbbR{Solving linear system \eqref{eq:KKTSys} is equivalent to the constrained optimization problem
	\begin{align} 
	    \min_{ \p \in \mathbb{R}^n } \quad & \frac{1}{2} \norm{\B\p}^2_2 - 
	    \rb_{k-1}\tp \A\p
	   \label{eq:itProbLong}
	\\ 
	   \text{ subject to } \quad & \jbbR{\Sbold\tp_{k}}\A(\jbbR{\x_{k-1}} + \p) = \jbbR{\Sbold\tp_{k}} \bbold. 
	   \label{eq:consLong}
	\end{align}
	When $\Sk \in \mathbb{R}^{n \times k}$ is in the range of $\A=\U\bs{\Sigma}\V\tp$, it can be represented as $\Sk=\U\b{T}_k$ using a 
	nonsingular matrix $\b{T}_k \in \mathbb{R}^{k \times k}$. Thus constraint \eqref{eq:consLong} implies
	\begin{equation*}
	    \bs{\Sigma}\V\tp(\x_{k-1}+\p)=\U\tp\b{b} \quad \text{ or } \quad \A\p=\U\U\tp\rb_{k-1,}
	\end{equation*}
	with $\rb_{k-1}=\bbold-\A\x_{k-1}$. Since $\rb_{k-1}\tp \A\p = \norm{\U\tp\rb_{k-1}}^2_2$ is constant with respect to $\p$, problem
	\eqref{eq:itProbLong}--\eqref{eq:consLong} is equivalently represented by
	\begin{align*} 
	    \min_{ \p \in \mathbb{R}^n } \quad & \frac{1}{2} \norm{\B\p}^2_2 
	\\ 
	   \text{ subject to } \quad & \jbbR{\Sbold\tp_{k}}\A(\jbbR{\x_{k-1}} + \p) = \jbbR{\Sbold\tp_{k}} \bbold. 
	\end{align*}}

	
\section{Linear combination of updates}
\label{app:B}
\jbb{We derive a representation of $ \p_{\jbbRt{k}} $ as a linear combination
of previous updates $ \biidx{p}{\jbbRt{k-1}},\ldots,\biidx{p}{1} $. Recall from \eqref{eq:compute-pk-alpha1}
that 
\begin{equation*}
    \p_{\jbbRt{k}} = \frac{\jbbRo{\jbbRt{\s_{k}}\tp \rb_{k-1}}}{\delta_{\jbbRt{k}}}\left( \y_{\jbbR{k}} - \Y_{k-1}\bhiidx{t}{k}   \right),
\end{equation*}
where 
\setlength{\arraycolsep}{2pt}
\begin{equation*}
   \bhiidx{t}{k} = \biidx{R}{k-1}\biidx{D}{k-1}\biidx{R}{k-1}\tp \biidx{Y}{k-1}\tp \y_{\jbbR{k}} =
    \bigg( 
        \bmat{\biidx{R}{k-2}\biidx{D}{k-2}\biidx{R}{k-2}\tp &
          \\                & 0} 
 + \frac{\biidx{\jbb{t}}{\jbbR{k-1}}\biidx{\jbb{t}}{\jbbR{k-1}}\tp}{\delta_{k-1}}
	    \bigg)\biidx{Y}{k-1}\tp \y_{\jbbR{k}}
\end{equation*}
\setlength{\arraycolsep}{5pt}
(with superscript indices suppressed).
From this we deduce
\begin{align*}
    \biidx{Y}{k-1}\bhiidx{t}{k} &= \biidx{Y}{k-2} \biidx{R}{k-2} \biidx{D}{k-2} \biidx{R}{k-2}\tp
    \biidx{Y}{k-2}\tp \y_{\jbbR{k}} + \frac{\biidx{t}{k-1}\tp(\biidx{Y}{k-1}\tp \y_{\jbbR{k}})}{\delta_{k-2}}\biidx{Y}{k-1}\biidx{t}{k-1}\\
	    &= \biidx{Y}{k-2} \biidx{R}{k-2} \biidx{D}{k-2} \biidx{R}{k-2}\tp
    \biidx{Y}{k-2}\tp \y_{\jbbR{k}} + \frac{\biidx{t}{k-1}\tp(\biidx{Y}{k-1}\tp \y_{\jbbR{k}})}{\delta_{k-1}} 
    (\biidx{Y}{k-2}\bhiidx{t}{k-1} - \biidx{y}{k-1}) \\
	    &= \biidx{Y}{k-2} \biidx{R}{k-2} \biidx{D}{k-2} \biidx{R}{k-2}\tp
    \biidx{Y}{k-2}\tp \y_{\jbbR{k}} - \frac{(\biidx{t}{k-1}\tp(\biidx{Y}{k-1}\tp \b{y}_{\jbbRt{k}}))}{\jbbRo{\s_{k-1}\tp \rb_{k-2}}}\biidx{p}{k-1}.
\end{align*}
Now for $j = 1 \To k-1$, define scalars 
$\alpha_{j} = \biidx{t}{j}\tp(\biidx{Y}{j}\tp \y_{\jbbR{k}}) / \jbbRo{(\s_{j}\tp \rb_{j-1})}$,
so that writing $\biidx{R}{k-2} \biidx{D}{k-2} \biidx{R}{k-2}$ 
in terms of $\biidx{R}{k-3} \biidx{D}{k-3} \biidx{R}{k-3}$ and $ \biidx{t}{k-2} $ (and recursively backwards) gives 
\begin{equation*}
    \biidx{Y}{k-1}\bhiidx{t}{k} = -\sum_{j=1}^{k-1} \alpha_{j} \biidx{p}{j}.
\end{equation*}
Hence the update $ \p_{\jbbRt{k}} $ can be represented by previous updates and the vector $ \y_{\jbbR{k}}$:
\begin{equation*}
    \p_{\jbbRt{k}} = \frac{\jbbRo{\jbbRt{\s_{k}}\tp \rb_{k-1}}}{\delta_{\jbbRt{k}}}\left( \y_{\jbbR{k}} - \Y_{k-1}\bhiidx{t}{k}   \right) = \frac{\jbbRo{\jbbRt{\s_{k}}\tp \rb_{k-1}}}{\delta_k} \big(  
    \jbbRt{\y_{k}} + \sum_{j=1}^{k-1} \alpha_{j} \biidx{p}{j} \big).
\end{equation*}}
	
\section{\jbrv{Condition numbers}}
\label{app:C}	
\jbrv{The condition numbers for the problems in Tables 1, 2 and 4 are shown in Table \ref{tab:cond}.}	

\begin{table}[h!]  
\captionsetup{width=1\linewidth}
\caption{\jbrv{Condition numbers $\kappa$ for the matrices in Tables 1, 2 and 4. Matrices in the first 7 rows
correspond to Tables 1 and 2. Those in rows 8--14 correspond to Table~4.}}
\label{tab:cond}

\setlength{\tabcolsep}{2pt} 

\scriptsize
\hbox to 1.00\textwidth{\hss 
\begin{tabular}{| l l | l l | l l | l l | l l | l l | } 
		\hline
        Matrix & $ \kappa $ &
        Matrix & $ \kappa $ &
        Matrix & $ \kappa $ &
        Matrix & $ \kappa $ &
        Matrix & $ \kappa $ &
        Matrix & $ \kappa $ \\
		\hline
        $\texttt{lpi\_gran}$ & 4e+13 &$\texttt{landmark}$ & 1e+08&$\texttt{Kemelmacher}$ & 2e+04&$\texttt{Maragal\_4}$ & 9e+06&$\texttt{Maragal\_5}$ & 5e+12&$\texttt{Franz4}$ & 5\\ 
$\texttt{Franz5}$ & 1e+01 &$\texttt{Franz6}$ & 8&$\texttt{Franz7}$ & 5&$\texttt{Franz8}$ & 6&$\texttt{Franz9}$ & 6&$\texttt{Franz10}$ & 6\\ 
$\texttt{GL7d12}$ & 8 &$\texttt{GL7d13}$ & 1e+10&$\texttt{ch6-6-b3}$ & 2&$\texttt{ch7-6-b3}$ & 2&$\texttt{ch7-8-b2}$ & 1&$\texttt{ch7-9-b2}$ & 1\\ 
$\texttt{ch8-8-b2}$ & 1 &$\texttt{cis-n4c6-b3}$ & 1&$\texttt{cis-n4c6-b4}$ & 1&$\texttt{mk10-b3}$ & 2&$\texttt{mk11-b3}$ & 2&$\texttt{mk12-b2}$ & 1\\ 
$\texttt{n2c6-b4}$ & 1 &$\texttt{n2c6-b5}$ & 1&$\texttt{n2c6-b6}$ & 2&$\texttt{n3c6-b4}$ & 1&$\texttt{n3c6-b5}$ & 1&$\texttt{n3c6-b6}$ & 1\\ 
$\texttt{n4c5-b4}$ & 1 &$\texttt{n4c5-b5}$ & 2&$\texttt{n4c5-b6}$ & 2&$\texttt{n4c6-b3}$ & 1&$\texttt{n4c6-b4}$ & 1&$\texttt{rel7}$ & 1e+01\\ 
$\texttt{relat7b}$ & 1e+01 &$\texttt{relat7}$ & 1e+01&$\texttt{mesh\_deform}$ & 1e+03&$\texttt{162bit}$ & 1e+03&$\texttt{176bit}$ & 3e+03&$\texttt{specular}$ & 3e+08\\
\hline
$\texttt{lp\_25fv47}$ & 3e+03 &$\texttt{lp\_bnl1}$ & 3e+03&$\texttt{lp\_bnl2}$ & 8e+03&$\texttt{lp\_cre\_a}$ & 2e+04&$\texttt{lp\_cre\_c}$ & 2e+04&$\texttt{lp\_cycle}$ & 1e+07\\ 
$\texttt{lp\_czprob}$ & 9e+03 &$\texttt{lp\_d2q06c}$ & 1e+05&$\texttt{lp\_d6cube}$ & 1e+03&$\texttt{lp\_degen3}$ & 8e+02&$\texttt{lp\_fffff800}$ & 1e+10&$\texttt{lp\_finnis}$ & 1e+03\\ 
$\texttt{lp\_fit1d}$ & 5e+03 &$\texttt{lp\_fit1p}$ & 7e+03&$\texttt{lp\_ganges}$ & 2e+04&$\texttt{lp\_gfrd\_pnc}$ & 9e+04&$\texttt{lp\_greenbea}$ & 4e+03&$\texttt{lp\_greenbeb}$ & 4e+03\\ 
$\texttt{lp\_ken\_07}$ & 1e+02 &$\texttt{lp\_maros}$ & 2e+06&$\texttt{lp\_maros\_r7}$ & 2&$\texttt{lp\_modszk1}$ & 4e+01&$\texttt{lp\_pds\_02}$ & 4e+01&$\texttt{lp\_perold}$ & 5e+05\\ 
$\texttt{lp\_pilot}$ & 3e+03 &$\texttt{lp\_pilot4}$ & 4e+05&$\texttt{lp\_pilot87}$ & 8e+03&$\texttt{lp\_pilot\_ja}$ & 3e+08&$\texttt{lp\_pilot\_we}$ & 5e+05&$\texttt{lp\_pilotnov}$ & 4e+09\\ 
$\texttt{lp\_qap12}$ & 3 &$\texttt{lp\_qap8}$ & 3&$\texttt{lp\_scfxm2}$ & 2e+04&$\texttt{lp\_scfxm3}$ & 2e+04&$\texttt{lp\_scrs8}$ & 9e+04&$\texttt{lp\_scsd6}$ & 9e+01\\ 
$\texttt{lp\_scsd8}$ & 1e+03 &$\texttt{lp\_sctap2}$ & 2e+02&$\texttt{lp\_sctap3}$ & 2e+02&$\texttt{lp\_shell}$ & 4e+01&$\texttt{lp\_ship04l}$ & 1e+02&$\texttt{lp\_ship04s}$ & 1e+02\\
   \hline
 \end{tabular}
 \hss}
\end{table}

\clearpage

	
\bibliographystyle{siamplain}
\bibliography{myrefs}
\end{document}